\documentclass[11pt]{amsart}
\usepackage[T1]{fontenc}
\usepackage{lmodern}
\usepackage{microtype}
\usepackage{amssymb}
\usepackage{amsthm}
\usepackage{amscd}
\usepackage{hyperref}
\usepackage{cite}
\usepackage{mathscinet}

\newtheorem{theorem}{Theorem}[section]

\newtheorem{lemma}[theorem]{Lemma}

\theoremstyle{definition}
\newtheorem{definition}[theorem]{Definition}

\newtheorem{example}[theorem]{Example}

\def\sgn{\operatorname{sgn}}

\makeatletter

\def\dotminussym#1#2{%
  \setbox0=\hbox{$\m@th#1-$}%
  \kern.5\wd0%
  \hbox to 0pt{\hss\hbox{$\m@th#1-$}\hss}%
  \raise.6\ht0\hbox to 0pt{\hss$\m@th#1.$\hss}%
  \kern.5\wd0}

\mathchardef\mhyphen="2D

\allowdisplaybreaks[2]

\begin{document}

\title{A Worked Example of the Functional Interpretation}
\author{Henry Towsner}
\date{\today}


\maketitle


\section{Introduction}

Suppose we have a \emph{non-constructive} proof: a proof that some quantity exists but which does not actually tell us what the value is.  Such a thing might happen because the proof proceeds by contradiction, because it uses the axiom of choice, or because the proof depends on an abstract technique which takes us far afield from the quantity we original set out to proof the existence of.

We might then ask whether there is any constructive proof of the statement---a proof which actually calculates the quantities in the statement.  The answer may be no.  Even if the answer is yes, it may require finding a completely different proof, which could be as hard as---or harder than---proving the theorem in the first place.

Our topic, however, is circumstances which guarantee that the answer is ``yes'': when knowing that there is a non-constructive proof guarantees the existence of a constructive proof, and further, tells us how to find it.  Strikingly, at least to those unfamiliar with proof theory, both the circumstances and the method for finding a constructive proof are \emph{syntactic}: they depend on the written form of the theorem and its proof.  While there exists more than one such method, our focus in this paper will be the \emph{functional interpretation} (also called the ``Dialectica'' interpretation, after the journal where it first appeared \cite{MR0102482}).

If we hold a strongly formal view of mathematics, we can ask for proper ``meta-theorems'': we could take the view that a proof in mathematics is always, at heart, a formal deduction in our favorite system of axioms---say, ZFC.  Then we could hope for a formal proof that, given any deduction of a theorem with the necessary syntactic property, there exists a constructive deduction of the same theorem.  For many choices of axioms (not currently including ZFC, but including systems which suffice to formalize most of mathematical practice), such meta-theorems actually exist.

These theorems are of limited applicability by themselves, however.  Even if one believes that actual proofs, as written in textbooks and articles, are intended as descriptions of formal deductions, obtaining those formal deductions is an arduous (and tedious) process.  The usefulness of the functional interpretation comes from the fact that it can be applied, with a feasible amount of effort, to actual proofs as written by and for humans.  Of course, we cannot have a theorem about ``all proofs accepted by mathematicians''; in the broader setting we have only a heuristic: the functional interpretation is a practical method for producing ordinary constructive proofs from ordinary non-constructive ones as already written up in journals.\footnote{Because proof theory is historically tied to intuitionist and formalist philosophical views, its dependence on these philosophies is sometimes overstated.  One need not believe that formal deductions, constructive proofs, or syntax play any special role in mathematics for proof theoretic methods to be useful.  Even someone who believes that axiomatic proofs are artificial constructs of no intrinsic importance should recognize that large swaths of mathematics just happen to be formalizable, and therefore that methods derived from their study just happen to be useful in practice.  Our attitude can be the one attributed to Niels Bohr regarding the lucky horseshoe above his door: ``I am told it works even if you don't believe in it''.}

As the title promises, our goal in this paper is to introduce this method by focusing on a concrete example.  In section \ref{sec:jacksons_thm} we'll present a non-constructive proof of a theorem on approximations of $L_1$ functions.  In section \ref{sec:quant} we'll present a corresponding quantitative theorem with a constructive proof.  These two sections will be completely elementary---all that's needed is some very basic real analysis.  (There is nothing new in our analysis of this theorem, which derives from Kohlenbach and Oliva's work \cite{MR1966746}.  We have chosen this example because the underlying theorem and proof are simple enough that they can be discussed in detail in a reasonable space.)

With a motivating example in place, in section \ref{sec:func_interp} we will finally introduce the functional interpretation, illustrating that the example in section \ref{sec:quant} is an instance of a general method.  This section will necessarily involve some actual formal logic---we'll assume some familiarity with ordinary first-order logic and the notion of computability, although we will keep the dependence as minimal as possible.

We will conclude with some references to further applications in the literature.

The author is grateful to Jeremy Avigad and Ulrich Kohlenbach for helpful suggestions on a draft of this paper.
 

\section{A Theorem of Jackson's}\label{sec:jacksons_thm}

Throughout, we will only be concerned with real-valued functions on $[0,1]$, which will refer to as just ``functions''.  When we write $\int f\,dx$ without bounds, we always mean the integral $\int_0^1 f\,dx$.

\begin{definition}
We say $f$ is an \emph{$L_1$-function} if $\int|f|\,dx$ exists.  The \emph{$L_1$-norm} is defined on such functions by
\[||f||_1=\int|f|\,dx.\]
\end{definition}
Many bizarre, complicated, or difficult-to-work-with functions are nonetheless $L_1$, so it is natural to ask about approximating such functions by nicer classes of functions; when the distance is measured using the $L_1$-norm, this is known as \emph{mean approximation}.  In particular:
\begin{definition}
Let $f$ be an $L_1$-function, and let $\mathcal{P}$ be a collection of $L_1$-functions.  A \emph{best (mean) approximation to $f$ in $\mathcal{P}$} is a function $p\in\mathcal{P}$ such that for every $q\in\mathcal{P}$,
\[||f-p||_1\leq ||f-q||_1.\]
\end{definition}
We will mostly be interested in approximation by polynomials of low degree.
\begin{definition}
  We write $\mathcal{P}_n$ for the collection of polynomials of degree at most $n$.
\end{definition}
When $\mathcal{P}$ is finite dimensional, the existence of a best approximation follows immediately from a theorem of Riesz \cite{MR1555146}.  Our focus will be on the question of whether the best approximation is \emph{unique}.  For an arbitrary set $\mathcal{P}$, it need not be:
\begin{example}
  Let $f$ be the function which is constantly equal to $2$, and let $\mathcal{P}$ consist of all piecewise continuous $L_1$-functions $p$ such that $||p||_1\leq 1$.  Clearly the function $p$ which is constantly equal to $1$ is a best approximation to $f$ in $\mathcal{P}$, but the function
\[q(x)=\left\{\begin{array}{ll}
2&\text{if }0\leq x\leq 1/2\\
0&\text{if }1/2<x\leq 1
\end{array}\right.\]
is \emph{also} a best approximation.
\end{example}
Uniqueness of approximations is related to a property known as \emph{strict convexity} of the norm (see, for instance, \cite{MR1656150} for more general theory on the subject of approximations); since the $L_1$-norm is not strictly convex, in general there is not a unique best mean approximation.  However for particular classes of functions, ad hoc arguments can still give uniqueness, and an example of such a result is due to Jackson \cite{MR1501177}:
\begin{theorem}
  Let $f$ be a continuous function.  Then for any $n$, there is a unique best approximation to $f$ in $\mathcal{P}_n$.
\end{theorem}

There are a number of proofs of this theorem \cite{MR0119004,MR0100752,MR0180848}, but we will examine a direct proof (in particular, avoiding the use of measure theory) due to Cheney \cite{MR0180789}, which we break into several lemmas to make the subsequent analysis easier.  (Certain equations are labeled because we will want to refer to them later.)

We need a few definitions.

\begin{definition}
  If $f$ is a continuous function, $||f||_\infty=\sup_{x\in[0,1]}|f(x)|$.
\end{definition}
(This is a simplification of the usual $L_\infty$ norm; we do not need the $L_\infty$ norm for discontinuous functions, and so can get away with this simplified definition.)  Note that, since $f$ is assumed to be continuous, $||f||_\infty$ is defined and finite.

\begin{definition}
  If $g$ is a function, we write $\sgn g$ for the \emph{sign} function, given by
\[\sgn g(x)=\left\{\begin{array}{ll}
1&\text{if }g(x)>0\\
0&\text{if }g(x)=0\\
-1&\text{if }g(x)<0
\end{array}\right..\]
\end{definition}

\begin{lemma}\label{non_zero_red_qual}
Let $g$ and $h$ be continuous functions such that $g$ has finitely many zeros and $\int h\sgn g\,dx\neq 0$.  Then for some $\lambda$, 
\[||g-\lambda h||_1<||g||_1.\]
\end{lemma}
\begin{proof}
Let $x_1,\ldots,x_k$ be the zeros of $g$.  We partition the interval $[0,1]$ into two sets, $A$ and $B$.  $B$ will consist of a small open interval around each $x_i$:
\[B=\bigcup_{i\leq k}(x_i-\epsilon,x_i+\epsilon)\cap[0,1]\]
while $A=[0,1]\setminus B$.  We will have to choose $\epsilon$ small enough.  Observe that
\begin{equation}\label{b4}
\int_B|h(x)|\, dx\leq ||h||_\infty\cdot k\cdot 2\epsilon,
\end{equation}
so by choosing $\epsilon$ very small, we may arrange for $\int_B|h(x)|\, dx$ to be as small as we need.

On the other hand, we know that $|\int h\sgn g\,dx|>0$, and we have
\begin{align}
\left|\int_A h\sgn g\,dx\right|  
&\geq \left|\int h\sgn g\,dx\right|-\left|\int_B h\sgn g\,dx\right|\notag\\
&\geq \left|\int h\sgn g\,dx\right|-\int_B |h|\, dx.\label{b3}
\end{align}
In particular, by choosing $\epsilon$ small enough that $2\int_B|h(x)|\, dx<\left|\int h\sgn g\,dx\right|$, we may arrange to have
\begin{equation}\label{b2}
\left|\int_A h\sgn g\,dx\right|>\int_B|h(x)|\, dx.
\end{equation}

$|g|$ is continuous, and therefore the restriction of $|g|$ to $A$ has a minimum somewhere on the closed set $A$, and since $g$ has no zeros in $A$, this infimum $m=\inf\{|g(x)|\mid x\in A\}$ must be positive.  Choose $\lambda$ so that $0<|\lambda| \cdot||h||_\infty<m$ and $\sgn\lambda=\sgn \left(\int_Ah\sgn g\,dx\right)$.  Then we have $|\lambda h(x)|<m\leq |g(x)|$ and $\sgn g=\sgn (g-\lambda h)$ for every $x\in A$.  Now we can compute:
\begin{align}
||g-\lambda h||_1
&=\int |g-\lambda h|\, dx\notag\\
&=\int_A|g-\lambda h|\, dx+\int_B|g-\lambda h|\, dx\notag\\
&=\int_A(g-\lambda h)\sgn g\,dx+\int_B|g-\lambda h|\, dx\notag\\
&=\int_A|g|\, dx-\lambda \int_A h\sgn g\,dx+\int_B|g-\lambda h|\, dx\notag\\
&=\int|g|\, dx-\int_B|g|\, dx-\lambda \int_A h\sgn g\,dx+\int_B|g-\lambda h|\, dx\notag\\
&\leq\int|g|\, dx-\lambda \int_A h\sgn g\,dx+\int_B|g|+|\lambda|\,|h| dx-\int_B|g|\, dx\notag\\
&=\int|g|\, dx-|\lambda| \left|\int_A h\sgn g\,dx\right|+|\lambda|\int_B|h|\, dx\notag\\
&=\int|g|\, dx-|\lambda|(\left|\int_A h\sgn g\,dx\right|-\int_B|h|\, dx)\label{b1}\\
&<\int|g|\, dx\notag\\
&=||g||_1\notag.
\end{align}
\end{proof}

\begin{lemma}\label{best_approx_zero_qual}
    If $f$ is continuous and $p$ is a best approximation to $f$ in $\mathcal{P}_n$ then $f-p$ has at least $n+1$ zeros.
\end{lemma}
\begin{proof}
Let $g(x)=f(x)-p(x)$.  Suppose the statement is false, so $g$ has at most $n$ zeros.  Since $f$, and therefore $g$, is continuous, there are at most $n$ points where $g$ changes sign.  For some $m\leq n$, we may choose points
\[0<r_1<r_2<\cdots<r_m<1\]
so that the $r_i$ are exactly the interior points where $g$ changes sign.  Consider the polynomial
\[h(x)=\prod_{i=1}^{m}(x-r_i).\]
Since $h(x)$ also changes sign exactly at the $r_i$, it follows that $h\sgn g$ is either always non-negative or always non-positive, and is not constantly $0$, and therefore
\begin{equation}
  \label{non_zero}
\int h\sgn g\,dx\neq 0.  
\end{equation}

By Lemma \ref{non_zero_red_qual}, there is a $\lambda$ such that $||g-\lambda h||_1<||g||_1$, and equivalently,
\[||f-(p+\lambda h)||_1=||g-\lambda h||_1<||g||_1.\]
But $p+\lambda h$ is a polynomial of degree at most $n$, contradicting the assumption that $p$ was a best approximation.
\end{proof}

In order to derive (\ref{non_zero}), we used the following seemingly trivial fact, which we will need again:
\begin{lemma}\label{cont_is_zero_qual}
 Let $q$ be a continuous function with $\int |q(x)|\, dx=0$.  Then $||q||_\infty=0$.
\end{lemma}

\begin{theorem}\label{main_qual}
  Let $f$ be a continuous function.  Then for any $n$, there is a unique best approximation to $f$ in $\mathcal{P}_n$.
\end{theorem}
\begin{proof}
  Suppose $p_1,p_2$ are both best approximations of $f$ by polynomials of degree at most $n$.  Let $p$ be the average, $p(x)=\frac{1}{2}(p_1(x)+p_2(x))$.  Clearly $p$ is also a polynomial of degree at most $n$.  Also 
  \begin{align*}
    ||f-p||_1
&=\int |f(x)-p(x)|\, dx\\
&=\int |f(x)-\frac{1}{2}(p_1(x)+p_2(x))|\, dx\\
&\leq\int \frac{1}{2}|f(x)-p_1(x)|+\frac{1}{2}|f(x)-p_2(x)|\, dx\\
&=\frac{1}{2}||f-p_1||_1+\frac{1}{2}||f-p_2||_1.
 \end{align*}
This means that $p$ is also a best approximation of $f$.

Since $p_1$ and $p_2$ are both best approximations, we have $||f-p||_1=||f-p_1||_1=||f-p_2||_1$, and therefore
\begin{align}
  0&=
\frac{1}{2}||f-p_1||_1+\frac{1}{2}||f-p_2||_1-||f-p||_1\notag\\
&=\int\frac{1}{2}|f(x)-p_1(x)|+\frac{1}{2}|f(x)-p_2(x)|-|f(x)-p(x)|\, dx.\label{c2}
\end{align}

On the other hand, by the triangle inequality, for every $x$,
\begin{align*}
|f(x)-p(x)|
&=|f(x)-\frac{1}{2}(p_1(x)+p_2(x))|\\
&=|\frac{1}{2}(f(x)-p_1(x))+\frac{1}{2}(f(x)-p_2(x))|\\
&\leq\frac{1}{2}|f(x)-p_1(x)|+\frac{1}{2}|f(x)-p_2(x)|.
\end{align*}
In particular, 
\begin{equation}
  \label{c1}
  \frac{1}{2}|f(x)-p_1(x)|+\frac{1}{2}|f(x)-p_2(x)|-|f(x)-p(x)|\geq 0
\end{equation}
 for every $x$.

Combining (\ref{c2}) with (\ref{c1}) using Lemma \ref{cont_is_zero_qual}, for all $x$
\[\frac{1}{2}|f(x)-p_1(x)|+\frac{1}{2}|f(x)-p_2(x)|-|f(x)-p(x)|=0.\]
Therefore for every $x$,
\[|f(x)-p(x)|=\frac{1}{2}|f(x)-p_1(x)|+\frac{1}{2}|f(x)-p_2(x)|.\]

Since $p$ is a best approximation to $f$, by Lemma \ref{best_approx_zero_qual}, $f-p$ must have at least $n+1$ zeros.

Suppose $r$ is a zero of $f-p$.  Then
\[0=|f(r)-p(r)|=\frac{1}{2}|f(r)-p_1(r)|+\frac{1}{2}|f(r)-p_2(r)|\]
and therefore
\[\frac{1}{2}|f(r)-p_1(r)|=\frac{1}{2}|f(r)-p_2(r)|=0.\]
Then $p_1(r)=f(r)=p_2(r)$, and therefore $p_1(r)-p_2(r)=0$.  $p_1-p_2$ is a polynomial of degree at most $n$ which is $0$ at every zero of $f-p$.  The only polynomial of degree at most $n$ which has $n+1$ zeros is $0$, so $p_1-p_2$ must be constantly $0$, and therefore $p_1=p_2$.
\end{proof}

\section{Rates of Unicity}\label{sec:quant}

Not all uniqueness results are created equal.  Once we know there is a unique approximation, we can ask for more detailed quantitative information about the approximation. 
\begin{definition}
$p\in\mathcal{P}_n$ is a \emph{$\epsilon$-nearly best approximation to $f$ in $\mathcal{P}_n$} if for every $q\in\mathcal{P}_n$,
\[||f-p||_1< ||f-q||_1+\epsilon.\]
\end{definition}
We know that two best approximations in $\mathcal{P}_n$ must be equal, so the analogous question to ask is whether $\epsilon$-nearly best approximations must be near each other.  More precisely, we look for a function $\Phi_f$ such that if $p_1$ and $p_2$ are $\Phi_f(\delta)$-nearly best approximations then $||p_1-p_2||_1<\delta$.  Such a function $\Phi_f$ is known as a \emph{modulus of uniqueness} for $f$.

It's clear from the definition that if there is any modulus of uniqueness then there are many functions meeting the definition (for instance, if $\Phi_f$ is a modulus of uniqueness for $f$ and $\Psi(\delta)\geq\Phi_f(\delta)$ for every $\delta$ then $\Psi$ is also a modulus of uniqueness for $f$).  However if there is a unique best approximation to $f$, we would expect that there is a ``nice'' modulus of uniqueness---one that is continuous, at least from one side, and where $\lim_{\delta\rightarrow 0}\Phi_f(\delta)=0$.  The optimal modulus has these properties under suitable conditions, and is sometimes called "the" modulus of uniqueness.  An arbitrary modulus of uniqueness, however, need not be so nice.  The particular case where $c\delta$ is a modulus of uniqueness for some constant $c$ is called \emph{strong unicity}, and is well-studied in approximation theory (again, see \cite{MR1656150}).

A second type of quantitative information is the stability of the mean approximation under small changes to the function.  Let us write $\mathcal{A}$ for the functional mapping $f$ to its best approximation in $\mathcal{P}$.  When $||f-f'||_1\leq\epsilon$, what can we say about $||\mathcal{A}(f)-\mathcal{A}(f')||_1$?  A modulus of uniqueness immediately answers this question.

\begin{theorem}
Suppose $f,f'$ have unique mean approximations in $\mathcal{P}$, $\Phi_f$ is the modulus of uniqueness for $f$, and $||f-f'||_1<\frac{1}{2}\Phi_f(\delta)$.  Then $||\mathcal{A}(f)-\mathcal{A}(f')||_1<\delta$.
\end{theorem}
\begin{proof}
  Suppose $||f-f'||_1<\frac{1}{2}\Phi_f(\delta)$.  
For any $q\in\mathcal{P}$,
  \begin{align*}
    ||f-\mathcal{A}(f')||_1
&\leq ||f-f'||_1+||f'-\mathcal{A}(f')||_1\\
&\leq ||f-f'||_1+||f'-q||_1\\
&\leq 2||f-f'||_1+||f-q||_1\\
&< ||f-q||_1+\Phi_f(\delta).
  \end{align*}
Therefore $\mathcal{A}(f')$ is a $\Phi_f(\delta)$-nearly best approximation to $f$, and since $\Phi_f$ is a modulus of uniqueness, 
\[||\mathcal{A}(f)-\mathcal{A}(f')||_1<\delta.\]
\end{proof}

Our goal in the remainder of this section is to find a modulus of uniqueness corresponding to Jackson's Theorem, closely following \cite{MR1966746}.  More precisely, we want to find a modulus corresponding to the \emph{particular proof} given in the previous section---we think of this as the process of making our arguments more quantitative: instead of working with best approximations, we'll work with nearly best approximations; instead of working with non-zero values, we'll include numerical bounds on how large the values must be; and so on.

We begin with a simple (but, as it turns out, essential) observation: if $||p||_1> 2||f||_1$ then $||f-0||_1<||f-p||_1$, and therefore when considering approximations, we can consider only those polynomials which satisfy $||p||_1\leq 2||f||_1$.

\begin{definition}
  $\mathcal{Q}_n\subseteq\mathcal{P}_n$ is the set of polynomials $p$ of degree $\leq n$ such that $||p||_1\leq 2||f||_1$.
\end{definition}

In particular, the best approximation to $f$ must belong to $\mathcal{Q}_n$.  From here on we will restrict our attention to $\mathcal{Q}_n$.

\subsection{A Quantitative Lemma \ref{non_zero_red_qual}}

We can expect that in order to obtain a quantitative version of the main theorem, we'll need quantitative versions of the intermediate lemmas as well.  We'll start, naturally, with Lemma \ref{non_zero_red_qual}; we need to strengthen the conclusion so that $||g-\lambda h||_1+\delta<||g||_1$ for some fixed value $\delta$.

Of course, the stronger conclusion isn't true unless we make stronger assumptions as well---in order to get a quantitative conclusion, we'll need quantitative assumptions.  We look through the proof to find out which steps contribute to the size of the gap between $||g||_1$ and $||g-\lambda h||_1$, and then take steps to bound those values.

The actual guarantee that the gap exists is given by (\ref{b1}), where we can see that the size of the gap is $\lambda(\left|\int_Ah\sgn g\,dx\right|-\int_B|h|\, dx)$.  This gives us two values we need to bound away from $0$: $\lambda$ and $\left|\int_Ah\sgn g\,dx\right|-\int_B|h|\, dx$.  The latter is shown to be positive at (\ref{b2}); from this and (\ref{b3}), we can see that there are two factors contributing to the size of $\left|\int_Ah\sgn g\,dx\right|-\int_B|h|\, dx$: the size of $\left|\int h\sgn g\, dx\right|$ and the size of $\int_B|h(x)|\, dx$.  The first of these is simply one of our assumptions: instead of making the qualitative assumption that $\int h\sgn g\,dx\neq 0$, we will make a quantitative assumption that $\left|\int h\sgn g\, dx\right|\geq\theta$ for some appropriate $\theta$.

We turn to the bound on $\int_B|h(x)|\, dx$.  The bound on this value is given by (\ref{b4}), where it is established that
\[\int_B|h(x)|\, dx\leq ||h||_\infty\cdot k\cdot 2\epsilon.\]
A bound on $||h||_\infty$ is just an additional assumption, as is a particular value for $k$, but $\epsilon$ is some value we have chosen.  As long as we are willing to make $\epsilon$ small, we can make $\int_B|h(x)|\,dx$ as small as we want.  So as we make $\epsilon$ smaller, we make $\int_B|h(x)|\,dx$ smaller, and therefore seem to make $||g||_1-||g-\lambda h||_1$ larger.  However when we turn to $\lambda$, we will see that making $\epsilon$ small may force $\lambda$ to be small, which makes $||g||_1-||g-\lambda h||_1$ small.

$\lambda$ was given by the rule that $\lambda ||h||_\infty<m=\inf\{|g(x)|\mid x\in A\}$, so to keep $\lambda$ bounded away from $0$, we need to bound $m$ away from $0$.  But here we run into a problem: what if $g$ has a point which dips very close to $0$, but isn't near any of the zeros?  Then this point is included in the set $A$, but it forces $m$ to be very small, which in turn forces $\lambda$ to be small, and therefore causes $||g-\lambda h||_1$ to be very close to $||g||_1$.

Suppose $g_\zeta$ is a function which ``almost'' has a zero---say, $g_\zeta(x)=(x-1/2)^2+\zeta$ where $\zeta$ is very small.  If we take the proof given in the previous section at face value, $\zeta$ being small just forces us to choose $\lambda_\zeta$ very small.  Indeed, following this argument, as $\zeta$ approaches $0$, the gap $||g_\zeta||_1-||g_\zeta-\lambda_\zeta h||_1$ approaches $0$ as well.

When $\zeta$ reaches $0$, however, the situation changes.  Now $1/2$ is a zero of $g_0(x)$, and $\lambda$ only has to be concerned with the value of $g_0(x)$ when $x$ is outside the interval $(1/2-\epsilon,1/2+\epsilon)$.  In particular, we can still choose a value $\lambda_0$ so that $||g_0||_1-||g_0-\lambda_0 h||_1>0$.  It turns out to be a general principle that quantitative results ought not exhibit such discontinuities---all parameters in our proof should vary continuously in the function $g$.  It is therefore not surprising that we can easily modify our proof to eliminate this discontinuity: when $\zeta$ is very small, we could make $||g_\zeta||_1-||g_\zeta-\lambda_\zeta h||_1$ larger by treating $1/2$ as if it were a zero and removing the interval around it.  This has a price (it makes $k$ larger, because we have a new zero, which in turn makes $\int_B|h(x)|\,dx$ larger, and therefore $\left|\int_Ah\sgn g\,dx\right|-\int_B|h|\,dx$ smaller), but when $\zeta$ is very small, we still obtain better bounds this way.

Instead of assuming that $g$ has finitely many zeros, we'll assume there are finitely many points where the value is below some parameter $\zeta$ (with the number of such points and the size of $\zeta$ contributing to our ultimate bound on $||g||_1-||g-\lambda h||_1$).  Stated like this, we still don't have the formulation quite right, since when $g$ is continuous, having any zeros at all will lead to having uncountably many points below $\zeta$; clearly having many small points ``close togther'' doesn't count.  We will remove a small interval around each almost-zero, so what matters---what forces the set $B$ to be large---is if there are many points close to $0$ which are separated by large gaps.  It is more convenient to jump to our ultimate purpose, and simply require that there is a set $B$ of measure at most $\epsilon$ containing all points $r$ with $|g(r)|<\zeta$.

We can now prove this quantitative formulation using essentially the same argument as in the proof of Lemma \ref{non_zero_red_qual}, but filling in explicit calculations where appropriate.
\begin{lemma}\label{non_zero_red_quant}
Let $g,h,\epsilon,\zeta$ be given such that:
  \begin{itemize}
  \item  $g$ and $h$ are $L_1$-functions,
  \item $\epsilon>0$ and $\zeta>0$,
  \item $|\int h\sgn g\,dx|\geq 3K\epsilon$,
  \item $||h||_\infty\leq K$,
  \item There is a set $B$ of measure at most $\epsilon$ such that if $|g(x)|<\zeta$ then $x\in B$.
  \end{itemize}
Then 
\[||g-\frac{\zeta}{2K} h||_1+\epsilon\frac{\zeta}{2}\leq ||g||_1.\]
\end{lemma}
\begin{proof}
  Let $A=[0,1]\setminus B$.  We have
\[\int_B|h|\,dx\leq K\epsilon.\]
We have
\[\left|\int_A h\sgn g\,dx\right|\geq \left|\int h\sgn g\,dx\right|-\int_B|h|\, dx\geq 2K\epsilon\]
and therefore
\[\left|\int_Ah\sgn g\,dx\right|-\int_B|h|\,dx\geq K\epsilon.\]

Note that if $x\in A$ then $|g(x)|\geq\zeta$.  Let $\lambda=\frac{\zeta}{2K}$.  Then for any $x\in A$, $|\lambda h(x)|\leq \frac{\zeta}{2}<\zeta\leq|g(x)|$.  In particular, $\sgn g(x)=\sgn(g(x)-\lambda h(x))$ for every $x\in A$.  Then, by the same calculations as in the proof of Lemma \ref{non_zero_red_qual}, 
\[||g-\lambda h||_1\leq ||g||_1-\lambda\left(\left|\int_Ah\sgn g\,dx\right|-\int_B|h(x)|\,dx\right)\leq ||g||_1-\epsilon\frac{\zeta}{2}.\]
\end{proof}

\subsection{A Quantitative Lemma \ref{best_approx_zero_qual}}
We now turn to formulating a quantitative version of Lemma \ref{best_approx_zero_qual}.  We would expect to need a quantitative version of Lemma \ref{cont_is_zero_qual} to derive (\ref{non_zero}).  In fact, a particularly nice version which suffices for our purposes is available in the existing literature:
\begin{theorem}[Markov brothers' Inequality\footnote{This inequality is sometimes called just ``Markov's inequality'', but as that name is also used for an inequality from probability theory, this name is perhaps less confusing.  The inequality given here was proven by Andrey Markov, the same one who gave his name to the other Markov's inequality.  His younger brother Vladimir, also a mathematician, proved a generalization to multiple derivatives, leading to the name given here, which, properly, applies only to the generalization.}]    If $p$ has degree $\leq n$, $||p'||_\infty\leq 2n^2||p||_\infty$.
\end{theorem}
An immediate consequence of this is
\begin{lemma}\label{markov2}
      If $p$ has degree $\leq n$, $||p||_\infty\leq 2(n+1)^2||p||_1$.
\end{lemma}
\begin{proof}
  Let $q(x)=\int_0^x p(y)\,dy$ be the integral of $p$.  Applying Markov's inequality to $q$ gives 
  \begin{align*}
||p||_\infty
&\leq 2(n+1)^2||q||_\infty\\
&=2(n+1)^2\sup_x \left|\int_0^x p(y)\,dy\right|\\
&\leq 2(n+1)^2\sup_x \int_0^x |p(y)|\,dy\\
&= 2(n+1)^2\int_0^1 |p(y)|\,dy\\
&= 2(n+1)^2||p||_1.\\
 \end{align*}
\end{proof}

As for stating a quantitative version of Lemma \ref{best_approx_zero_qual}, our work in the previous subsection gives us a good guess what will need to happen: we will need to replace the assumption that $p$ is a best approximation with the assumption that $p$ is a nearly best approximation, and the conclusion that $f-p$ has $n$ zeros with the conclusion that $f-p$ has a collection of $n$ ``almost'' zeros which are well spread out.

\begin{lemma}\label{best_approx_zero_quant}
  Suppose $f$ is continuous and $p$ is a $\frac{\zeta}{20(n+2)^2}$-nearly best approximation to $f$ in $\mathcal{Q}_n$.  Then there are $n+1$ points $r_1<\cdots<r_{n+1}$ such that $r_{i+1}-r_i\geq \frac{1}{20(n+2)^2n}$ for $i<n+1$ and $|f(r_i)-p(r_i)|\leq \zeta$ for each $i\leq n+1$.
\end{lemma}
\begin{proof}
Let $\theta=\frac{1}{20(n+2)^2n}$, $\epsilon=\frac{1}{10(n+2)^2}$, and $g(x)=f(x)-p(x)$.  Suppose the conclusion is false; then there must be some $k\leq n$ and some $r'_1<\cdots<r'_{k}$ such that $|g(r'_{i})|\leq\zeta$ for each $i\leq k$, but whenever $|g(r)|\leq\zeta$ for any $r\in[0,1]$, there is some $i\leq k$ such that $|r-r'_i|< \theta$.  In particular, the set $B=\bigcup_{i\leq k}(r'_i-\theta,r'_i+\theta)$ contains all points $r$ with $|g(r)|\leq\zeta$.  Observe that $\mu(B)\leq 2n\theta=\frac{1}{10(n+2)^2}=\epsilon$.

If $r'_i+\theta<r'_{i+1}-\theta$ then, since $g(x)$ is continuous, $g$ does not change sign on the interval $[r'_i+\theta,r'_{i+1}-\theta]$.  So all sign changes take place in some interval $(r'_i-\theta,r'_i+\theta)$.  We choose a subsequence $r_1,\ldots,r_m$ and consider the polynomial
\[h'(x)=\prod_{i=1}^m(x-r_i).\]
By choosing the subsequence $r_i$ appropriately, we can ensure that $h(x)=(-1)^bh'(x)$ has the same sign as $g(x)$ whenever $x\not\in B$.  (Saying this precisely is tricky: roughly, we want to include $r'_i$ in our list if $\sgn g(r'_i-\theta)\neq \sgn g(r'_i+\theta)$, but the situation is complicated by the fact that we could have $r'_{i+1}-\theta<r'_{i+1}+\theta)$ and have the sign change occur somewhere in that interval; in that case we would want to include either of $r'_i$ and $r'_{i+1}$, but not both.)  Let $K=||h||_\infty$; by Lemma \ref{markov2}, $||h||_1\geq \frac{K}{2(n+1)^2}$.  

On the other hand, we have
\[\left|\int_B h\sgn g\,dx\right|\leq\int_B|h(x)|\,dx\leq K\mu(B)=K\epsilon,\]
and, taking $A=[0,1]\setminus B$,
\[||h||_1=\int|h(x)|\,dx=\int_A|h(x)|\,dx+\int_B|h(x)|\,dx\]
Putting these together gives
\[\int_A|h(x)|\,dx\geq \frac{K}{2(n+1)^2}-\frac{K}{10(n+1)^2}=\frac{2K}{5(n+1)^2}.\]
  Therefore
\begin{align*}
\left|\int h\sgn g\,dx\right|
&\geq\left|\int_A h\sgn g\,dx\right|-\int_B|h(x)|\,dx\\
&=\int_A|h(x)|\,dx-\int_B|h(x)|\,dx\\
&\geq \frac{2K}{5(n+1)^2}-\frac{K}{10(n+1)^2}\\
&=\frac{3K}{10(n+1)^2}\\
&\geq 3K\epsilon.
\end{align*}

Therefore we may apply Lemma \ref{non_zero_red_quant} and conclude that there is a $\lambda$ so that $||g-\lambda h||_1+\epsilon\frac{\zeta}{2}\leq||g||_1$, contradicting the assumption that $p$ was a $\frac{\zeta}{20(n+1)^2}$-nearly best approximation.
\end{proof}

\subsection{A Quantitative Lemma \ref{cont_is_zero_qual}}

We now turn to Lemma \ref{cont_is_zero_qual}.  We replaced one use of this lemma with the Markov brothers' inequality, but there is a second use, in the proof of the main theorem, which cannot be replaced by the Markov brothers' inequality (since that only applies to polynomials).  In its qualitative form, this lemma seems like an obvious fact about integrals, but we now need a quantitative version.  A quantitative formulation should weaken the assumption to merely $\int |q(x)|\, dx\leq\epsilon$ for some bound $\epsilon$.  Of course, we can no longer hope to have $q(x)=0$ for every $x$, since we can easily think of continuous functions which have small but non-zero integrals but are not always $0$.  Therefore we expect to weaken the conclusion as well---instead of having $||q||_\infty=0$, we might hope to show that $||q||_\infty<\delta$ for some value $\delta$.

We can easily imagine continuous functions $q$ with $||q||_1$ small but $||q||_\infty$ quite large: $q$ could be $0$ except on a small ``bump'' where it gets very large.  By making the bump very narrow, we can let the bump be very tall, causing $||q||_\infty$ to be large even though $||q||_1$ is staying small.  These bumps don't create discontinuity, but it's not unreasonable to say that a function with a tall, thin bump is ``less'' continuous than functions without such features.  We can quantify this effect:
\begin{definition}
  Let $q$ be a continuous $L_1$ function.  A \emph{modulus of uniform continuity} of $q$ is a function $\omega_q(\epsilon)$ such that for any $\epsilon>0$ and any $x,y\in[0,1]$ such that $|y-x|<\omega_q(\epsilon)$, $|q(y)-q(x)|<\epsilon$.
\end{definition}

It's not hard to see that the modulus of uniform continuity is sufficient to give a quantitative version of the lemma: if $|q(x)|\geq\epsilon$ for some $x$, the modulus of uniform continuity ensures that $|q(y)|\geq\epsilon/2$ when $y$ is near $x$, and this ensure that $\int |q(x)|\, dx$ cannot be too small.  To make this precise:
\begin{lemma}\label{cont_is_zero_quant}
  Let $q$ be a function with modulus of continuity $\omega_q$ and
\[\int |q|\,dx<\frac{\epsilon}{2}\cdot\min\{\frac{1}{2},\omega_q(\epsilon/2)\}.\]
Then $||q||_{\infty}<\epsilon$.
\end{lemma}
\begin{proof}
We prove the contrapositive.  If $|q(x)|\geq\epsilon$ then for every $y$ with $|y-x|<\omega_q(\epsilon/2)$, $|q(y)|>\epsilon/2$.  We cannot be sure the whole interval $(x-\omega_q(\epsilon/2),x+\omega_q(\epsilon/2))$ belongs to $[0,1]$, but setting $\eta=\min\{\frac{1}{2},\omega_q(\epsilon/2)\}$, we can be sure that if $x<1/2$ then $(x,x+\eta)\subseteq[0,1]$, while if $x\geq 1/2$ then $(x-\eta,x)\subseteq[0,1]$.  The two cases are symmetric, so consider the case where $x<1/2$, and therefore
\begin{align*}
\int |q(x)|\,dx
&=\int_0^{x}|q(x)|\,dx+\int_x^{x+\eta}|q(x)|\,dx+\int_{x+\eta}^1|q(x)|\,dx\\
&> 0+(\epsilon/2)\eta+0\\
&\geq(\epsilon/2)\eta.
\end{align*}
\end{proof}

Now when we apply Lemma \ref{cont_is_zero_quant}, we need to know the modulus of continuity for the function we apply the lemma to.  The only application in our proof not covered by the Markov brothers' inequality is to the function
\[\frac{1}{2}|f(x)-p_1(x)|+\frac{1}{2}|f(x)-p_2(x)|-|f(x)-p(x)|.\]

$f$ is a function given to us; it is no surprise that our bounds will depend on modulus of continuity of $f$, so we'll include as an assumption that we are given particular bounds on the modulus of continuity of $f$.

For the polynomials $p_1,p_2$, we can use a combination of the Markov brothers' inequality and the fact that we restricted our polynomials to those in $\mathcal{Q}_n$ to obtain a modulus of continuity.  (This is where we use the fact that we are optimizing over $\mathcal{Q}_n$ instead of $\mathcal{P}_n$.)
\begin{lemma}
  Let $p\in\mathcal{Q}_n$.  Then $\omega_p(\epsilon)=\frac{\epsilon}{4n^2(n+1)^2||f||_1}$ is a modulus of uniform continuity for $p$.
\end{lemma}
\begin{proof}
$p$ is differentiable, so for any $x<y$ in $[0,1]$,
\[|p(y)-p(x)|\leq\left|\int_x^y p'(z)\,dz\right|\leq (y-x)||p'||_{\infty}.\]
Applying the Markov brothers' inequality to $p'$, we have
\[||p'||_{\infty}\leq 2n^2||p||_\infty\leq 2n^2(n+1)^2||p||_1\leq 4n^2(n+1)^2||f||_1.\]
Therefore for any $x<y$,
\[|p(y)-p(x)|\leq 4n^2(n+1)^2||f||_1(y-x)\]
and so the function $\omega_p(\epsilon)=\frac{\epsilon}{4n^2(n+1)^2||f||_1}$ is a modulus of uniform continuity for $p$.
\end{proof}

The function we ultimately need is a linear combination of absolute values of functions.  It suffices to observe the following:
\begin{lemma}\label{modulus_of_cont_combination}
  Let $f,g$ be functions, let $c$ be a constant, and let $\omega_f,\omega_g$ be corresponding moduli of uniform continuity.  Then:
  \begin{enumerate}
  \item $\omega_f$ is a modulus of uniform continuity for $|f(x)|$,
  \item $\omega_{cf}(\epsilon)=\omega_f(\epsilon/|c|)$ is a modulus of uniform continuity for $cf(x)$,
  \item $\omega_{f+g}(\epsilon)=\min\{\omega_f(\epsilon/2),\omega_g(\epsilon/2)\}$ is a modulus of uniform continuity for $f(x)+g(x)$.
  \end{enumerate}
\end{lemma}
\begin{proof}
  \begin{enumerate}
  \item Suppose $|y-x|<\omega_f(\epsilon)$.  Then $\left||f(y)|-|f(x)|\right|\leq |f(y)-f(x)|<\epsilon$.
  \item Suppose $|y-x|<\omega_f(\epsilon/c)$.  Then $|cf(y)-cf(x)|=|c|\cdot |f(y)-f(x)|<|c|\epsilon/|c|=\epsilon$
\item   Suppose $|x-y|<\min\{\omega_f(\epsilon/2),\omega_g(\epsilon/2)\}$.  Then
\[|(f+g)(x)-(f+g)(y)|\leq|f(x)-f(y)|+|g(x)-g(y)|<\epsilon/2+\epsilon/2=\epsilon.\]
  \end{enumerate}
\end{proof}
By the same argument, we may replace the last item in the lemma with $\omega_{f+g+h}(\epsilon)=\min\{\omega_f(\epsilon/3),\omega_g(\epsilon/3),\omega_h(\epsilon/3)\}$ for the sume of three functions.

\subsection{A Quantitative Version of Jackson's Theorem}
Finally, we turn to the main theorem, Theorem \ref{main_qual}.  We know that we need to strengthen the assumption by stipulating a modulus of continuity for the function $f$, and that in exchange we expect to calculate a modulus of uniqueness for the approximation of $f$ in $\mathcal{Q}_n$.

Once again, we go through the proof systematically adding quantitative information to statements.  Suppose that instead of being best approximations, $p_1$ and $p_2$ are $\epsilon$-nearly best approximations.  Our goal is to obtain some kind of bound on $||p_1-p_2||_1$.  The function $p(x)=\frac{1}{2}(p_1(x)+p_2(x))$ is no longer a best approximation, but it is still an $\epsilon$-nearly best approximation.  As a result, the function
\[q(x)=\frac{1}{2}|f(x)-p_1(x)|+\frac{1}{2}|f(x)-p_2(x)|-|f(x)-p(x)|\]
is still non-negative, and has a small (but not necessarily $0$) integral.

The main difference is that we will ultimately obtain points $r_1,\ldots,r_{n+1}$ with the property that $|p_1(r)-p_2(r)|$ is small, but not necessarily $0$.  In the qualitative version, we used the presence of $n+1$ zeros, and the fact that $p_1-p_2$ was a polynomial of degree at most $n$, to conclude that $p_1-p_2$ was constantly $0$.  A quantitative version should allow us to conclude from the presence of $n+1$ ``well-separated almost-zeros'' that $p_1-p_2$ is close to the constantly $0$ polynomial.

We know that $n+1$ points are enough to specify a polynomial of degree at most $n$, so if we write down \emph{any} polynomial of degree at most $n$ going through all the points $(r_i,p_1(r_i)-p_2(r_i))$, we will have written down the polynomial $p_1-p_2$.  We just need to express this polynomial in a way that makes explicit that the polynomial is always small.  As it happens, one of the most common ways of writing down a polynomial from its zeros has precisely this property: the \emph{Lagrange interpolant} is given by
\[L(x)=\sum_{i=1}^{n+1}\left(p_1(r_i)-p_2(r_i)\right)\prod_{j\neq i}\frac{x-r_j}{r_i-r_j}.\]
Observe that for any $k$, for each $i\neq k$ the term $\prod_{j\neq i}\frac{r_k-r_j}{r_i-r_j}$ is $0$, and therefore
\[L(r_k)=\left(p_1(r_k)-p_2(r_k)\right)\prod_{j\neq k}\frac{r_k-r_j}{r_k-r_j}=p_1(r_k)-p_2(r_k),\]
so the Lagrange interpolant does goes through the desired points.  It is also clear that we have retained exactly enough information about the points $(r_i,p_1(r_i)-p_2(r_i))$ to place a bound on $||L||_{\infty}$: we know that each $p_1(r_i)-p_2(r_i)$ is small and that as long as $j\neq i$, $r_j-r_i$ is not too small.

Stated formally, we have the following theorem:
\begin{theorem}\label{main_quant}
Let $f$ be a function with modulus of continuity $\omega_f$.  Then for any $n$, 
\[\Phi_f(\epsilon)=\frac{\zeta}{2}\min\{\frac{1}{10(n+2)^2},\omega_f(\zeta/12),\frac{\zeta}{24n^2(n+1)^2||f||_1}\}\]
where $\zeta=\frac{\epsilon}{4(n+1)20^n(n+2)^{2n}n^n}$, is a modulus of uniqueness for the approximation by $\mathcal{Q}_n$.
\end{theorem}
\begin{proof}
We fix some values: $\gamma=\frac{1}{20(n+2)^2n}$, $\upsilon=\frac{\epsilon\gamma^n}{n+1}$, $\zeta=\upsilon/4$, and $\rho=\frac{\zeta}{2}\min\{\frac{1}{10(n+2)^2},\omega_f(\zeta/12),\frac{\zeta}{24n^2(n+1)^2||f||_1}\}$.

Let $p_1,p_2$ be $\rho$-nearly best approximations to $f$ in $\mathcal{Q}_n$.  Let $p(x)=\frac{1}{2}(p_1(x)+p_2(x))$.  $p$ is also an $\rho$-nearly best approximation to $f$ since
\[||f-p||_1\leq \frac{1}{2}||f-p_1||_1+\frac{1}{2}||f-p_2||_1.\]
This means the equation (\ref{c2}) gets weakened to
\[\int\frac{1}{2}|f(x)-p_1(x)|+\frac{1}{2}|f(x)-p_2(x)|-|f(x)-p(x)|\, dx<\rho.\]
But (\ref{c1}),
\[\frac{1}{2}|f(x)-p_1(x)|+\frac{1}{2}|f(x)-p_2(x)|-|f(x)-p(x)|\geq 0,\]
 is still valid.

Let
\[q(x)=\frac{1}{2}|f(x)-p_1(x)|+\frac{1}{2}|f(x)-p_2(x)|-|f(x)-p(x)|.\]
We wish to apply Lemma \ref{cont_is_zero_quant}, which requires a modulus of continuity.  We have the modulus $\omega_f$ for $f$ and the modulus $\omega_p(\delta)=\frac{\delta}{4n^2(n+1)^2||f||_1}$ for $p,p_1,$ and $p_2$, so we may apply Lemma \ref{modulus_of_cont_combination} to obtain the modulus of continuity
\[\omega_q(\delta)=\min\{\omega_f(\delta/6),\frac{\delta}{12n^2(n+1)^2||f||_1}\}.\]
In particular, $\rho\leq\frac{\zeta}{2}\min\{1/2,\omega_q(\zeta/2)\}$, so by Lemma \ref{cont_is_zero_quant}, we must have $||q||_\infty<\zeta$.

$\rho\leq\frac{\zeta}{20(n+2)^2}$, so by Lemma \ref{best_approx_zero_quant}, we have $n+1$ points $r_1<\cdots<r_{n+1}$ such that $r_{i+1}-r_i\geq \gamma$ for $i<n+1$ and $|f(r_i)-p(r_i)|\leq \zeta$ for each $i\leq n+1$.  Therefore for each $i\leq n+1$,
\begin{align*}
  |p_1(r_i)-p_2(r_i)|
&\leq |f(r_i)-p_1(r_i)|+|f(r_i)-p_2(r_i)|\\
&\leq 2\left[q(r_i)+|f(r_i)-p(r_i)|\right]\\
&< 2[\zeta+\zeta]\\
&=\upsilon.
\end{align*}

The Lagrange interpolant gives us an expression for the polynomial $p_1(x)-p_2(x)$:
\[p_1(x)-p_2(x)=\sum_{i=1}^{n+1}(p_1(r_i)-p_2(r_i))\prod_{j\neq i}\frac{x-r_j}{r_i-r_j}.\]
We may bound this:
\begin{align*}
  |p_1(x)-p_2(x)|
&= \left|\sum_{i=1}^{n+1}(p_1(r_i)-p_2(r_i))\prod_{j\neq i}\frac{x-r_j}{r_i-r_j}\right|\\
&\leq \sum_{i=1}^{n+1}\left|p_1(r_i)-p_2(r_i)\right|\prod_{j\neq i}\frac{|x-r_j|}{|r_i-r_j|}\\
&< \sum_{i=1}^{n+1} \upsilon\gamma^{-n}\\
&=n\upsilon\gamma^{-n}\\
&=\epsilon,
\end{align*}
and therefore
\[||p_1-p_2||_1=\int |p_1-p_2|\, dx< \epsilon.\]
\end{proof}

An additional merit of this proof is then we have obtained information on the \emph{uniformity} of the bounds: we now know if that if someone gives us only $\omega_f$ and $||f||_1$, we can compute a modulus $\Phi$ which will work for \emph{any} function $g$ for which $\omega_f$ is a modulus of continuity and with $||g||_1\leq ||f||_1$.

\section{The Functional Interpretation}\label{sec:func_interp}

We have now seen that it was possible not only to prove an explicit quantitative version of Jackson's Theorem, but to find such a proof closely related to the qualitative proof.  Still, the arguments we used appeared to be somewhat ad hoc, and we seem to have gotten lucky in various places.  We made a string of fortunate guesses, for instance replacing $\mathcal{P}_n$ with $\mathcal{Q}_n$, or strengthening the assumption of Lemma \ref{non_zero_red_quant}, and therefore weakening the conclusion of Lemma \ref{best_approx_zero_quant}, in a way which just so happened to allow us to finish the proof anyway.  We also depended on existing arguments, like the Markov brothers' inequality and the properties of the Lagrange interpolant, to complete our quantitative proof.

Our goal in this section is to describe a formalization of the procedure used in the previous section which applies systematically to certain kinds of proofs.  (Indeed, this method is precisely the way the arguments in the previous section were found by Kohlenbach and Oliva.)

In order to make this precise, we need to pin down what counts as a quantitative calculation.  For our purposes, we will identify ``quantitative'' with ``computable''.  Then our goal in this section is to describe how to systematically take a proof of a statement and:
\begin{enumerate}
\item Find a related statement for which it is appropriate to investigate the existence of computable bounds,
\item Convert the proof into a calculation of those computable bounds.
\end{enumerate}

As mentioned in the introduction, one of the things that makes the functional interpretation useful in practice is that we can take two different perspectives on it.  In the first perspective, the functional interpretation is a completely formal idea; it refers to a family of theorems of the following form:
\begin{theorem}
  Suppose that there is a proof of $\phi$ in the formal theory $T$.  Then there is a proof of $\phi^{ND}$ in the formal theorem $T'$.
\end{theorem}
Here $T$ is some particular formal theory, $T'$ is a theory related to $T$ but with the additional property that proofs in $T'$ are constructive, and $\cdot^{ND}$ is an operation mapping formulas of first-order logic to formulas in some particular nice form.  The main reason we view these theorems as instances of a uniform idea is that the transformation $\phi\mapsto\phi^{ND}$ is very similar across all these theorems.  Such theorems have been proven for a variety of theories $T$.

The second perspective is that the functional interpretation is a heuristic: across many theories and situations, there is an operation which takes any formula $\phi$ (in particular, where there may be no computable bounds at all for $\phi$) and transforms $\phi$ to $\phi^{ND}$, a statement for which computable bounds do exist, which is reliably \emph{implication preserving}---if we can conclude $\psi$ from $\phi$ then we can also conclude $\psi^{ND}$ from $\phi^{ND}$.  Since we can convert individual mathematical statements into formulas $\phi$ without too much difficulty, we can also convert mathematical statements into the quantitative formula $\phi^{ND}$.

Taking the first perspective would require choosing a suitable formal theory $T$, often the theory of Peano arithmetic or a variant, and carefully formalizing our statement in this theory.  This tends to involve a great deal of tedious coding---one must interpret statements about the real numbers as statements about sequences of integers, statements about integration as formulas involving quantification over partitions, and so on.  Nonetheless, the formal approach is often useful, and can give insights that the informal approach cannot.

Here, however, we will take the second approach, and work only semi-formally.  We will use the notation of first-order logic, but deal somewhat informally with our exact choice of formal language and theory.

\subsection{Formalizing Statements}
The first thing we need to do is translate ordinary mathematical language into the more formal language of first-order logic.  As promised, we will work semi-formally, writing formulas using quantifiers $\forall,\exists$ and the connectives $\wedge$ (and), $\vee$ (or), and $\rightarrow$ (implies), but without pinning down an exact language.  In particular, we will freely write things like $\forall q\in\mathbb{Q}\,\exists n\in\mathbb{N}\ldots$ to indicate quantification over various particular domains.

In order to get meaningful results, we do need some restrictions on the formulas we use.  (These restrictions stand in for actually formalizing a statement in a particular fixed language---essentially, the restrictions we are choosing are the ones which ensure that a proper formalization is possible.)  The most important restriction is that we only quantify over countable domains.  Since Jackson's theorem concerns notions like functions on the reals, this seems like a significant limitation.  We will work around this by taking advantage of the fact that many of the uncountable domains we are interested in (like the reals) are separable, and therefore many statements can be approximated by quantifying over the countable dense subset.

We will write $\mathbb{Q}^+$ for $\mathbb{Q}\cap(0,\infty)$, which is one such countable domain.  We will write $\mathcal{P}_n^{\mathbb{Q}}$ for the polynomials of degree at most $n$ with rational coefficients (which, of course, are dense in $\mathcal{P}_n$ under an appropriate topology).

In addition, we need a restriction on our atomic formulas.  The correct restriction is that our atomic formulas should represent only computable operations; making sense of that formally would require giving computable interpretations to things like real numbers and functions on real numbers.  We will take a short-cut: our atomic formulas $\phi(x)$ (where $x$ is one or more free variables) will always have the form $f(x)<\epsilon$ or $f(x)\leq \epsilon$ where $f$ is some continuous function.  (This is closely related to \emph{continuous logic}\cite{MR2436146}.)

Having just imposed the requirement that we work with countable domains, we will immediately allow a single exception: we allow a single, outermost universal quantifier over an uncountable domain.  That is, we work with formulas of the form
\[\forall {X\mathrel{\in}\mathcal{U}}\ \phi(X)\]
where $\mathcal{U}$ may be uncountable, but all quantifiers in $\phi$ must be over countable domains.\footnote{This is actually a harmless modification.  Formally, we could augment our language of first-order logic by a new predicate or function symbol, $X$.  Since there are no defining axioms for $X$, proving a formula $\phi(X)$ is equivalent to proving $\forall X\in\mathcal{U}\ \phi(X)$.}

To illustrate the idea, let us consider how we translate the uniqueness part of  Jackson's theorem into a formula of the specified form.  In English, the uniqueness part of Jackson's theorem says
\begin{quote}
  Let $f$ be a continuous $L_1$-function on $[0,1]$.  Then for any $n$, there is at most one best approximation to $f$ in $\mathcal{P}_n$.
\end{quote}
The collection of all continuous $L_1$-functions is uncountable, but we can include $f$ in the outermost uncountable quantifier.  We will have to write formulas involving the values of $x$, so we want to be able to quantifier over the domain of $f$; we can take the outermost quantifier over $\mathbb{R}^{\mathbb{Q}\cap[0,1]}$, since a continuous function is determined by its values at the rationals.

We first need a formula $\mathrm{cont}(f)$, which should hold exactly when $f$ is continuous.  The usual $\epsilon\mhyphen\delta$ formulation of continuity would be appropriate, but since $[0,1]$ is compact, continuity is equivalent to uniform continuity, which will slightly simplify things later.  So we take
\[\mathrm{cont}(f)\equiv\forall\epsilon\in\mathbb{Q}^+\,\exists\delta\in\mathbb{Q}^+\,\forall x,y\in\mathbb{Q}\cap[0,1]\,\left(|x-y|<\delta\rightarrow | f(x)- f(y)|<\epsilon\right).\]
Saying $f$ is $L_1$ is easy:
\[\mathrm{L_1}(f)\equiv\exists M\int |f(x)|dx< M.\]
Note that we insist on expanding $\mathrm{cont}(f)$ into the $\epsilon\mhyphen\delta$ form but are willing to treat $\int|f(x)|dx< M$ as a single statement, rather than writing it out in terms of a partition.  

We would like the conclusion to say that if $p_1,p_2$ are best approximations of $f$ then $p_1=p_2$.  But the real best approximation might not have rational coefficients, and we only want to quantify over $\mathcal{P}_n^{\mathbb{Q}}$.  We therefore want to reformulate the statement that there is a best approximation in terms of $\mathcal{P}_n^{\mathbb{Q}}$.

Suppose there were distinct best approximations; then there would be $p_1,p_2$ with $||p_1-p_2||_1>\epsilon$ for some $\epsilon$.  Since each of $p_1,p_2$ are arbitrarily well approximated by elements of $\mathcal{P}_n^{\mathbb{Q}}$, we could find $p'_1,p'_2\in\mathcal{P}_n^{\mathbb{Q}}$ with $||p'_i-p_1||_1$ arbitrarily small---and therefore $p'_i$ a $\delta$-nearly best approximation for $\delta$ as small as we want---and with $||p'_1-p'_2||_1>\epsilon$.  Therefore we should state that this does not happen:
\begin{align*}
\mathrm{approx}(f)
&\equiv\forall\epsilon\in\mathbb{Q}^+\ \exists\delta\in\mathbb{Q}^+\ \forall p'_1\in\mathcal{P}_n^{\mathbb{Q}}\ \forall p'_2\in\mathcal{P}_n^{\mathbb{Q}}\\
&\ \ \ [||p'_1-p'_2||_1>\epsilon\rightarrow\exists p'\in\mathcal{P}_n^{\mathbb{Q}}\\
&\ \ \ \ \ \ \ (||f-p'||_1+\delta\leq ||f-p'_1||_1\\
&\ \ \ \ \ \ \ \vee||f-p'||_1+\delta\leq ||f-p'_2||_1)].
\end{align*}

Putting this all together, the uniqueness part of Jackson's theorem is:
\begin{equation}
\forall f\in\mathbb{R}^{\mathbb{Q}\cap[0,1]}\left(\mathrm{cont}(f)\wedge\mathrm{L_1}(f)\rightarrow\mathrm{approx}(f)\right).\label{jackson_first_form}\end{equation}

\subsection{Extracting Quantitative Statements}

Once we have placed a formula in the form
\[\forall x\exists y\,\phi(x,y)\]
where $x$ and $y$ may be tuples of multiple variables, there is a natural way to identify a potential quantitative analog of this formula: replace the existence of $y$ with some calculated value---$\forall x\,\phi(x,g(x))$ for some reasonable function $g$.  In practice, computing exact values is often messy, so it usually enough to settle for some kind of bound: $\forall x\exists y\in G(x)\,\phi(x,y)$, where $G(x)$ is always some sort of bounded (really, compact) set.  In the simplest, but representative, case, $y$ ranges over the natural numbers and $G(x)$ always has the form $[0,n]$, so $G(x)$ really gives a bound on the size of $y$

The functional interpretation depends on the fact that whether or not there is a computable $g$ is closely related to the syntactic properties of $\phi$: if every quantifier in $\phi$ is over a compact domain (and our formalization was appropriate) then there is guaranteed to be a computable $g$.  Conversely, if there is a computable bound $g$ then there must be some formula equivalent to $g$ which contains only quantifiers over compact domains.  (The one complication to this equivalence is that when $\phi$ has quantifiers over non-compact domains, there is no easy way to tell whether it is equivalent to a simpler formula).

To see why this should be the case, consider the simplest situation: $x$ and $y$ are natural numbers and $\phi$ is a formula of \emph{arithmetic}, where the only domain being quantified over is the natural numbers.  If all quantifiers in $\phi$ are over compact subsets of the natural numbers then all quantifiers in $\phi$ are really over finite sets.  The atomic formulas should themselves be computable, and therefore there is a computer program which, given values $n,m$, checks in finite time whether or not $\phi(n,m)$ is true.  If $\forall x\exists y\phi(x,y)$ is true then $g$ is a computable function: given the input $n$, $g(n)$ first checks whether $\phi(n,0)$ holds; if so, $g(n)$ returns $0$, and otherwise, $g(n)$ checks whether $\phi(n,1)$ holds, if so returns $1$, and if not continues similarly.  (The standard encoding of computability in arithmetic gives the other half of a correspondence---any computable function can be converted into a formula of the right form.)

When we deal with more complicated spaces, as in the previous subsection, we want to distinguish between compact and non-compact quantifiers even while both are countable.  For instance, we usually want quantifiers over $\mathbb{Q}$ to be non-compact while quantifiers over $\mathbb{Q}\cap[0,1]$ should be seen as compact.  So in place of quantifiers over compact domains, we sometimes need quantifiers over countable dense subsets of compact domains.  This could be ambiguous---every set is dense in \emph{some} compact space (indeed, the one-point compactification adds only a single point); however a choice of topology is enforced by our choice of atomic formulas, which would not be continuous with respect to the wrong compactification.

For our purposes, we adopt the following convention: we call a set \emph{effectively compact} if it is a countable dense subset of a compact set (where the atomic formulas we are interested in come from functions which are continuous with respect to the same topology).  In practice, the only effectively compact sets we are concerned with are obvious ones: $\mathbb{Q}\cap[a,b]$ and the set of polynomials of degree $\leq n$ with coefficients from $\mathbb{Q}\cap[a,b]$.  These are clearly dense in the corresponding compact sets $[a,b]$ and the polynomials of degree $\leq n$ with coefficients from $[a,b]$, respectively.

Note that effectively compact sets allow finite searches the same way finite sets did: suppose $\mathcal{D}$ is the underlying, uncountable, compact space and $\mathcal{Q}\subseteq\mathcal{D}$ is a countable dense subset, and we want to check whether $\exists x\in\mathcal{D}\phi(c,x)$ holds for some fixed constants $c\in\mathcal{D}$.  We can choose a finite set $F\subseteq\mathcal{Q}$ which is sufficiently dense (based on the moduli of uniform continuity of the atomic formulas in $\phi$) and check, for each $d\in F$, whether $\phi(c,d)$ holds.

A formula in the form $\forall x\exists y\phi(x,y)$ where $\phi$ only has quantifiers over effectively compact domains is called a \emph{$\Pi_2$ formula}.  (The $\Pi$ indicates that the outermost quantifier is $\forall$ and the $2$ indicates that there are two blocks of quantifiers over non-compact domains.)

Let's consider what this means for Jackson's Theorem.  The formula \eqref{jackson_first_form} we found above isn't in the $\Pi_2$ form yet.  Recall that in the previous section we replaced the space $\mathcal{P}_n$ with the space $\mathcal{Q}_n$.  Now we see what motivated this change: while the space $\mathcal{P}_n$ is not compact, $\mathcal{Q}_n$ is compact; the corresponding $\mathcal{Q}_n^{\mathbb{Q}}$ is a countable set dense in the compact separable space $\mathcal{Q}_n$, so by replacing the quantifiers over $\mathcal{P}_n^{\mathbb{Q}}$ with quantifiers over $\mathcal{Q}_n^{\mathbb{Q}}$, we bring Jackson's theorem closer to the $\Pi_2$ form.  In fact, with this change, the conclusion of the formula is now in the right form:
\begin{align*}
&\forall\epsilon\in\mathbb{Q}^+\ \exists\delta\in\mathbb{Q}^+\ \forall p'_1\in\mathcal{Q}_n^{\mathbb{Q}}\ \forall p'_2\in\mathcal{Q}_n^{\mathbb{Q}}\\
&\ \ \ [||p'_1-p'_2||_1>\epsilon\rightarrow\exists p'\in\mathcal{Q}_n^{\mathbb{Q}}\\
&\ \ \ \ \ \ \ (||f-p'||_1+\delta\leq ||f-p'_1||_1\\
&\ \ \ \ \ \ \ \vee||f-p'||_1+\delta\leq ||f-p'_2||_1)].
\end{align*}
We define the formula:
\begin{align*}
\mathrm{approxQ}(f,\epsilon,\delta)
&\equiv\forall p'_1\in\mathcal{Q}_n^{\mathbb{Q}}\ \forall p'_2\in\mathcal{Q}_n^{\mathbb{Q}}\\
&\ \ \ [||p'_1-p'_2||_1>\epsilon\rightarrow\exists p'\in\mathcal{Q}_n^{\mathbb{Q}}\\
&\ \ \ \ \ \ \ (||f-p'||_1+\delta\leq ||f-p'_1||_1\\
&\ \ \ \ \ \ \ \vee||f-p'||_1+\delta\leq ||f-p'_2||_1)].
\end{align*}

To deal with the assumptions $\mathrm{cont}(f)$ and $\mathrm{L_1}(f)$, we use a technique called \emph{Skolemization}.  The idea is to replace statements like
\[\forall x\in\mathbb{Q}\,\exists y\in\mathbb{Q}\,\phi(x,y)\]
with the equivalent statement
\[\exists Y\in\mathbb{Q}^{\mathbb{Q}}\,\forall x\in\mathbb{Q}\,\phi(x,Y(x)).\]
In other words, we can move existential quantifiers outwards by replacing them with functions.  For example, we can rewrite $\mathrm{cont}(f)$ as:
\[\exists\omega\in(\mathbb{Q}^+)^{\mathbb{Q}^+}\forall\epsilon\in\mathbb{Q}^+\forall x,y\in\mathbb{Q}\cap[0,q](|x-y|<\omega(\epsilon)\rightarrow|f(x)-f(y)|<\epsilon).\]
$\omega$ should look familar: this is just the statement that $\omega$ is a modulus of uniform continuity for $f$.  We write
\[\mathrm{ucont}(f,\omega)\equiv\forall\epsilon\in\mathbb{Q}^+\forall x,y\in\mathbb{Q}\cap[0,q](|x-y|<\omega(\epsilon)\rightarrow|f(x)-f(y)|<\epsilon).\]
Both $\exists\delta\in\mathbb{Q}^+\mathrm{approxQ}(f,\epsilon,\delta)$ and $\neg\mathrm{ucont}(f,\omega)$ have the form $\exists y\phi(y)$ where $\phi(y)$ has only quantifiers over effectively compact domains.  Then we can rewrite Jackson's theorem as
\begin{align*}
  \forall f\in\mathbb{R}^{\mathbb{Q}\cap[0,1]}\forall \omega_f\in(\mathbb{Q}^+)^{\mathbb{Q}^+}\forall M\forall\epsilon\in\mathbb{Q}^+[&\neg\mathrm{ucont}(f,\omega)\vee\\
&\int|f(x)|dx\geq M\vee\\
&\exists\delta\in\mathbb{Q}^+\mathrm{approxQ}(f,\epsilon,\delta)].
\end{align*}
So we have converted Jackson's theorem into a $\Pi_2$ form.  (Notice that Skolemization doesn't trivialize the importance of the $\Pi_2$ form, because the $\exists$ quantifier in the $\Pi_2$ form is still supposed to range over a countable domain.)

Without going any further, we've already learned something: we expect there to be a computable function which, given $f$, $\omega_f$, $M$ and $\epsilon$, computes a $\delta$ such that $\mathrm{approxQ}(f,\epsilon,\delta)$.  Actually, since we don't use anything about $f$ other than $\mathrm{ucont}(f)$ and $\int|f(x)|dx<M$, we don't expect the bound to depend on $f$: in other words, we expect the modulus of uniqueness to depend only on $\omega_f$ and $M$.  (In fact, as discussed in \cite{MR1966746}, more careful examination of this statement shows that we can do a bit better: by replacing $f$ by $\tilde f(x)=f(x)-f(0)$, we can eliminate the dependence of $\int|f(x)|dx$, so that bounds depend \emph{only} on the modulus of continuity.)




\subsection{The Significance of Syntax}

It is worth stating explicitly what the previous subsection implied: if a theorem can be put into the $\Pi_2$ form $\forall x\exists y\phi(x,y)$, we expect there to be a computable function $G$ so that $G(x)$ is always effectively compact and such that we can prove that $\forall x\exists y\in G(x)\,\phi(x,y)$.  Furthermore, as we will describe in the next subsection, we expect to be able to take a proof of the original statement and systematically convert it into a particular choice of $G$ and a proof of the bounded statement.

If a theorem \emph{cannot} be put into the $\Pi_2$ form then this might not happen.  It is possible that we can both prove $\forall x\exists y\forall z\phi(x,y,z)$ and also prove that the value of $x$ does not suffice to give a computable bound on the value of $y$.  Further, it may be very hard to determine whether this is the case.

If we want to think about which proofs have constructive bounds and which don't, we have to begin with the fact that there is a qualitative difference between $\Pi_2$ statements and other statements.

\subsection{Proof Extraction}

So suppose we have a proof of a $\Pi_2$ statement, as in Jackson's Theorem.  We have said that we expect there to be a computable bound, and we now turn to the question of extracting such a bound from a proof of the original statement.

Let us begin by considering what we could do with completely formal proofs---that is, proofs which consist of a sequence of formulas, and with each step justified by some formal axiom or inference rule.  (We won't worry too much about the exact rules allowed; any reasonable choice will do.)  If every formula in the proof were in the $\Pi_2$ form then we would expect the proof to directly provide an explicit algorithm: typically any $\Pi_2$ axiom will have an obvious computable bound, and standard inference rules all preserve the property of having explicit computable bounds.

However many proofs of $\Pi_2$ statements have intermediate steps which are \emph{not} $\Pi_2$.  Jackson's theorem is an example: the statement of Lemma \ref{non_zero_red_qual} is not $\Pi_2$.  (Of course, this statement requires substantial massaging for it to even be meaningful to ask whether the statement is $\Pi_2$, and showing that a statement cannot, in any way, be rearranged into a $\Pi_2$ form is much harder than showing that it can be.  But the germ of the idea is given in our discussion before Lemma \ref{non_zero_red_quant}: the bounds are discontinuous.)  Such statements break the flow of explicit computations through our proof, and must usually be replaced by $\Pi_2$ statements if we want to recover explicit bounds.

We'll focus on Lemma \ref{non_zero_red_qual}, since it was in the process of quantifying that lemma that we had to do the most work.  We face some difficulty translating the assumption that $g$ has at most $n$ zeros into our format.  A first attempt at a translation would look something like this:
\[\forall g,h\left[\left(\exists n\exists x_1,\ldots,x_n\forall y \left[g(y)=0\rightarrow\exists i\leq n\ y=x_i\right]\right)\rightarrow\mathrm{rest}(g,h)\right]\]
where $\mathrm{rest}(g,h)$ is a formalization of ``if $h$ is continuous and $\int h\sgn g\ dx\neq 0$ then for some $\lambda$, $||g-\lambda h||_1<||g||_1$''.  Standard manipulations on first-order logic allow us to pull some of the quantifiers to the front:
\[\forall g,h,n,x_1,\ldots,x_n\exists y\left[\left(g(y)=0\rightarrow\exists i\leq n\ y=x_i\right)\rightarrow\mathrm{rest}(g,h)\right].\]
The $x_1,\ldots,x_n$ are real numbers which are part of the initial block of quantifiers over uncountable domains, but $y$ is also an arbitrary real number.  We need to reformulate this so that we only need to consider rational values of $y$; we can't simply restrict the quantifier to rationals, though, since it could well be that the zero itself actually occurs at a real.

We need to make the statement more quantitative; the trick is to start in the right place.  To say that $y=x_i$ is equivalent to saying that for every $\delta$, $|y-x_i|<\delta$ and similarly for $g(y)=0$.  So we can rephrase this as:
\[\forall g,h,n,x_1,\ldots,x_n\exists y\left[\left(\left[\forall\zeta>0\ |g(y)|<\zeta\right]\rightarrow\exists i\leq n\forall\delta>0 \ |y-x_i|<\delta\right)\rightarrow\mathrm{rest}(g,h)\right].\]
We can now pull the quantifiers over $\zeta$ and $\delta$ out to the front; we actually get to choose which quantifier will be the outermost, and we should make this choice with the goal of minimizing the alternatations of quantifiers, and therefore the ultimate complexity of the statement:
\[\forall g,h,n,x_1,\ldots,x_n\exists y\forall\zeta>0\exists\delta>0\left[\left(|g(y)|<\zeta\rightarrow\exists i\leq n \ |y-x_i|<\delta\right)\rightarrow\mathrm{rest}(g,h)\right].\]
In this statement, it makes no difference whether we allow $y$ to range over the real numbers or the rationals, so we choose the rationals to make this a sentence of the allowed kind.  (This is really a reflection of the fact that we have chosen the correct atomic formulas---equality on reals is not computable, and therefore should not be treated as an atomic formula.)

Since $y$ now quantifies over an effectively compact domain (the rationals in the interval $[0,1]$), we can put it back inside the other quantifiers:
\[\forall g,h,n,x_1,\ldots,x_n\forall\zeta>0\exists\delta>0\left[\left(\forall y\ |g(y)|<\zeta\rightarrow\exists i\leq n \ |y-x_i|<\delta\right)\rightarrow\mathrm{rest}(g,h)\right].\]

To complete the process of finding the form of Lemma \ref{non_zero_red_quant}, we expand $\mathrm{rest}(g,h)$.  Initially, we might have:
\begin{align*}
  \forall g,h,\omega_g,\omega_hn,x_1,\ldots,x_n\forall\zeta>0\forall \gamma>0\bigg[
&\exists\delta>0\neg\left(\forall y\ |g(y)|<\zeta\rightarrow\exists i\leq n \ |y-x_i|<\delta\right)\\
&\vee \neg \mathrm{ucont}(g,\omega_g)\\
&\vee \neg\mathrm{ucont}(h,\omega_h)\\
&\vee |\int h\sgn g\ dx|< \gamma\\
&\vee \exists\lambda ||g-\lambda h||_1<||g||_1.
\end{align*}
As noted above, $\neg\mathrm{ucont}(f,\omega)$ has the $\exists y\phi(y)$ form, so this is a $\Pi_2$ formula.

The actual Lemma \ref{non_zero_red_quant} stated above includes some ad hoc simplification---replacing the particular witnesses $x_1,\ldots,x_n$ and $\delta$ with the measurable set $B=\bigcup_{i\leq n}(x_i-\delta,x_i+\delta)$, and replacing the continuity of $g$ and $h$ with just an $L_1$ bound on $h$.  These all weaken the assumptions, and therefore strengthen the overall theorem; they are discovered by observing what properties actually get used in the proof.  (Note that while the proof of the resulting statement goes through with the weakened assumptions---in particular, without the continuity of $g$---we actually used the continuity of $g$ to derive the equivalence of the modified statement with the original one.)

Most importantly, the syntactic manipulations have guided us into discovering that we should replace counting zeros with a bound on those $y$ with $|g(y)|<\zeta$.  What the proof-theoretic methods tells us is that this is guaranteed to be the right thing to look at: both that we will be able to obtain computable bounds for this lemma by using such a restriction, and that the rest of our proof will ultimately match up with this choice.  In particular, this is a \emph{local} transformation: while proving a quantitative version of Lemma \ref{non_zero_red_qual}, we don't need to look at the quantitative part of the other lemmas to know whether the proof will still work.

\subsection{More Complicated Sentences}

In the case of Jackson's Theorem, the statements of all our lemmas naturally unravel to $\Pi_2$ sentences, even if they didn't start that way.  In a more general proof (see below for references to some examples), that might not happen.  For instance, an intermediate step of the proof might be to show a statement of the form
\[\forall x\exists y\forall z\phi(x,y,z)\]
(if all quantifiers in $\phi$ are over effectively compact domains then this is called a $\Pi_3$ formula) and use this to prove 
\[\forall u\exists v\psi(u,v).\]
The proof might proceed something like this:
\begin{quote}
  Given $u$, calculate suitable values of $x$.  Then there is a $y$ such that $\forall z\phi(x,y,z)$, and we can use this $y$ to calculate a value $v$ such that $\psi(u,v)$ holds.
\end{quote}
However, since there might be no way to compute $y$ from $x$, this might not actually give us a computable algorithm.

The essential idea of the functional interpretation is to replace the non-quantitative formula $\forall x\exists y\forall z\phi(x,y,z)$ with a new quantitative analog which may be strictly weaker than the original version, but which is still sufficient to carry out the proof.  This is best illustrated by an example at the purely syntactic level: we will replace 
\[\forall x\in\mathbb{X}\,\exists y\in\mathbb{Y}\,\forall z\in\mathbb{Z}\,\phi(x,y,z)\]
 with
\[\forall x\in\mathbb{X}\,\forall Z\in\mathbb{Z}^{\mathbb{Y}}\,\exists y\in\mathbb{Y}\,\phi(x,y,Z(y)).\]
In other words, given $x$ and a \emph{function} $Z$, we can find a $y$ which works, not necessarily for every $z$, but at least for the particular value $Z(y)$ returned by the function.  Ultimately we will need to restrict $Z$ to be a computable function, though $\Pi_3$ statements are simple enough that it does not matter.  In fact, for statements this simple, the modified form is actually equivalent to the original:
\begin{lemma}
  $\forall x\exists y\forall z\phi(x,y,z)$ holds iff $\forall x\forall Z\exists y\phi(x,y,Z(y))$ holds.
\end{lemma}
\begin{proof}
  The left to right direction is obvious: if $\forall x\exists y\forall z\phi(x,y,z)$ then, given $x$, we let $y$ be the corresponding witness, and then $\forall z\phi(x,y,z)$ holds, so in particular, $\phi(x,y,Z(y))$.

  For the other direction, we show the contrapositive.  Suppose $\exists x\forall y\exists z\neg\phi(x,y,z)$.  Take some value $x$ such that $\forall y\exists z\neg\phi(x,y,z)$, and consider the function $Z$ which, given $y$, choose a value $Z(y)$ such that $\neg\phi(x,y,Z(y))$.  Then $x,Z$ form a counterexample to $\forall x\forall Z\exists y\phi(x,y,Z(y))$.
\end{proof}
Notice that in the left to right direction, our argument was constructive, in the sense that, given a value of $y$ which worked on the left, we obtained a value which worked on the right.  The right to left direction, however, was not constructive---it was a proof by contradiction.  Knowing a particular value of $y$ satisfying the formula of right, or even knowing a general method for finding, from each $x$ and $Z$, a value of $y$ is not enough to find the single value of $y$ which works on the left.

The statement 
\[\forall x\forall Z\exists y\phi(x,y,Z(y))\]
\emph{is} $\Pi_2$, so we expect to be able to concretely calculate $y$ from $x$ and $Z$.  More surprising is that when we have a proof of a $\Pi_2$ statement $\forall u\exists v\psi(u,v)$ from $\forall x\exists y\forall z\phi(x,y,z)$, we also have a proof of $\forall u\exists v\psi(u,v)$ from $\forall x\forall Z\exists y\phi(x,y,Z(y))$, and the bounds on $v$ (as a function of $u$) depend only on the bounds on $y$ as a function of $x$ and $Z$.  In other words, the statement $\forall x\forall Z\exists y\phi(x,y,Z(y))$ captures all the computable information present in the original statement.

To handle even more complicated sentences, with yet more alternations of quantifiers, we need not only functions, but \emph{functionals}---operations which map functions to functions, and then operations which map functionals to functionals, and so on.  In order to keep ourselves to countable domains (and also to meet our goal of working with quantitative data), in general we need to restrict ourselves to computable functions.  If the domains in the original statement are all countable (and coded appropriately) then it makes sense to talk about computable functions from a countable domain to another countable domain, and there are only countably many such functions.  (The outer quantifier over an uncountable domain becomes an \emph{oracle}---we fix an object from an uncountable domain, but then all further discussion is computable relative to the use of that object fixed at the beginning.)

To each formula $\phi$ we will assign a new formula, $\phi^{ND}$ which will always have the form
\[\exists y\forall x\,\phi_D(x,y),\]
where $\phi_D$ will only have quantifiers over effectively compact domains.  The intention is that we will have a systematic method for converting proofs of $\phi$ into a particular choice of $y$ together with a proof that $\forall x\phi_D(x,y)$ holds.  The definitions of $\phi_D$ and $\phi^{ND}$ are given by induction on the form of $\phi$.

Remember that we are now restricting ourselves to computable functions, so when we write $\mathbb{X}^{\mathbb{Y}}$ in the following definition, we mean the domain of computable functions from $\mathbb{Y}$ to $\mathbb{X}$.

When $\phi$ is atomic, $\phi_D=\phi$.  For the inductive case, suppose we already have $\phi_D(x,y)$ and $\psi_D(u,v)$ where $x$ has type $\mathbb{X}$, $y$ has type $\mathbb{Y}$, $u$ has type $\mathbb{U}$, and $v$ has type $\mathbb{V}$.  Then
\begin{enumerate}
\item $(\phi\wedge\psi)_D(x,u,y,v)$ is $\phi_D(x,y)\wedge\psi_D(u,v)$ and $(\phi\wedge\psi)^{ND}=\exists y,v\forall x,u\,\phi_D(x,y)\wedge\psi_D(u,v)$,
\item $(\neg\phi)_D(X,y)$ is $\neg\phi_D(X(y),y)$ and $(\neg\phi)^{ND}$ is $\exists X\in\mathbb{X}^{\mathbb{Y}}\forall y\neg\phi_D(X(y),y)$,
\item $(\phi\rightarrow\psi)_D(X,V,y,u)$ is $\phi_D(X(y,u),y)\rightarrow\psi_D(u,V(y))$ and $(\phi\rightarrow\psi)^{ND}$ is 
\[\exists X\in\mathbb{X}^{\mathbb{Y}\times\mathbb{U}}\exists V\in\mathbb{V}^{\mathbb{Y}}\forall y,u(\phi_D(X(y,u),y)\rightarrow\psi_D(u,V(y)),\]
\item $(\forall z\in S\,\phi)_D(x,Y,z)$ is $\phi_D(x,Y(z),z)$ and $(\forall z\in S\,\phi)^{ND}$ is $\exists Y\forall x,z\phi_D(x,Y(z),z)$.
\end{enumerate}
Instead of defining cases for $\vee$ and $\exists$, we derive them using the de Morgan laws: $(\exists z\phi)^{ND}$ is $(\neg\forall z\neg\phi)^{ND}$ and $(\phi\vee\psi)^{ND}$ is $(\neg(\neg\phi\wedge\neg\psi))^{ND}$.

It is generally useful to think of $\phi^{ND}$ as saying ``$y$ is a demonstration that $\phi$ is true'' where a purely mechanical verification that a statement is true would require checking that $\phi_D(x,y)$ holds for all possible values of $x$.  Then some of the inductive cases are easy to interpret: for instance, if $y$ demonstrates $\phi$ holds and $v$ demonstrates that $\psi$ holds then the pair $v,y$ demonstrates $\phi\wedge\psi$.

A more interesting case is implication; what does it mean to have an explicit demonstration for $\phi\rightarrow\psi$?  It means we should have an algorithm which converts demonstrations of $\phi$ into demonstrations of $\psi$.  This is the function $V$: given a $y$ which demonstrates $\phi$, $V(y)$ is a demonstration of $\psi$.  The complication is that when $y$ fails to be a demonstration of $\phi$---there exists some $x$ with $\neg\phi_D(x,y)$---we don't want $V(y)$ to be arbitrary.  Instead we want to have the property that not only can we determine $V(y)$ from $y$, but when $u$ is a counter-example to $V(y)$, we can find a counter-example $X(y,u)$ to $y$.  Because it is the most important case, it is worth dwelling on this point: the interpretation of $\rightarrow$ here requires that we have an algorithm which converts \emph{any} $y$ into a value $V(y)$, without needing to know whether $y$ actually works or not---the commitment is that \emph{if} $y$ works then $V(y)$ works as well, and furthermore that we know how to translate a counter-example to $V(y)$ into a counter-example to $y$ itself.  This imposes something like a continuity requirement on $V$: if $y$ is ``almost right'', in the sense that counter-examples are rare (for instance, the only counter-examples are very large values of $x$) then $V(y)$ should be ``almost right'' as well.

We would like to add the following clause which would cause effectively compact domains to behave more like finite sets:
\begin{quote}
  When $S$ is a effectively compact domain, $(\forall z\in S\,\phi)_D(x,Y,z)$ is $\exists y\leq Y\,\phi_D(x,y,z)$ and $(\forall z\in S\,\phi)^{ND}$ is $\exists Y\forall x\forall z\in S\exists y\leq Y\,\phi_D(x,y,z)$.
\end{quote}
This says, roughly, that when we quantify over a effectively compact domain, we obtain a bound on all witnesses needed for all elements of that domain.  Making this idea work turns out to be a bit difficult---this is the subject of the \emph{monotone} \cite{MR1195271} and \emph{bounded} \cite{MR2156133} functional interpretations.  However for most practical applications, the additional effort is justified.

There are two properties that are needed to make the $ND$ translation useful.  The first is that if we have a proof of $\phi$ in a reasonable theory (for instance, Peano arithmetic) then we have a particular value of $y$ together with a proof that $\forall x\phi_D(x,y)$ actually holds.  To prove this, we could pin down a particular formal system of axioms and inference rules, and then prove that the translation of each axiom is justified, and also that each inference rule preserves being justified.  It is essential that the proof itself is completely explicit---given a formal proof $\phi$, there is an explicit procedure for converting it into a proof of $\phi^{ND}$ together with the associated algorithm.  It is convenient that the algorithm is \emph{short}: a single line in the original proof generally translates to a fixed, small number of lines in the new proof (the exact values, of course, depending on the particular formal system).

The second important property of the $ND$ translation is that it does not alter $\Pi_2$ formulas:
\begin{lemma}
When $\phi$ is quantifier-free, $(\forall x\exists y\phi(x,y))^{ND}$ is equivalent to the existence of a computable function $Y$ such that $\forall x\phi(x,Y(x))$.
\end{lemma}
\begin{proof}
Since $\phi$ has no quantifiers, $\phi^{ND}$ is actually the formula $\phi$ except that $\vee$ has been replaced by $
  \neg(\neg\cdots\wedge\neg\cdots)$, which is equivalent by the de Morgan law.  Then simply plugging in the definitions above, $(\forall x\exists y\phi(x,y))^{ND}$ is equivalent to
\[\exists Y\forall x\phi(x,Y(x))\]
where $Y$ is a computable functionl.
\end{proof}

This, of course, is exactly what we want: given a proof of $\forall x\exists y\phi(x,y)$, we translate it, step by step, to an explicit proof of $(\forall x\exists y\phi(x,y))^{ND}$, which means we have an actual function $Y$ such that $\forall x\phi(x,Y(x))$.  However the intermediate steps of the proof have been changed to bring out explicit information which may have been hidden.

If we used the monotone functional interpretation in place of the $ND$ translation described above, we would get a stronger result; essentially
\begin{lemma}
  When $\phi$ only has quantifiers over effectively compact domains, $(\forall x\exists y\phi(x,y))^{ND}$ is equivalent to the existence of a computable function $Y$ such that $\forall x\exists y\leq Y(x)\,\phi(x,y)$.
\end{lemma}
Making precise what $y\leq Y(x)$ means when $y$ comes from a domain of complicated functions, however, is a bit tricky.

We've described this as a process applied to completely formal proofs.  This is not a very practical approach; real proofs, as written in journals, are far from being strings of logical formulas.  If we had to first translate those proofs into formal strings of logical inferences, that alone would be a huge process, and empirical experience suggests that the proof becomes several times as long when reduced to a completely formal proof\footnote{The ratio of length of the formalized proof to the informal one is known as the \emph{de Bruijn factor}, and a value of $4$ seems to be common\cite{MR2463993}.}.  What makes the functional interpretation useful for actual substantial proofs is that it can be applied relatively directly to journal proofs.

The functional interpretation is a local transformation: it tells us how to translate each individual statement.  So we can translate particular statements---say, the statements of individual lemmas---and then fill in the proofs by hand, knowing that there is a proof (and one close to the original).  If this proves too difficult, we can simply break the proof in half at some convenient point, translate the statement of the halfway point, and prove two shorter lemmas.

We need the functional interpretation to complete our goal that extraction of bounds depends on the syntactic form of the conclusion.  No matter what the intermediate steps look like, we can use the $ND$ translation to convert every step of the proof into an argument with explicit quantitative bounds.







\section{Some Applications of the Functional Interpretation}\label{sec:applications}

\subsection{Fixed Point Theorems}

One place where non-constructive proofs occur naturally is metric spaces (often special kinds of metric spaces, like Banach spaces or $C^*$-algebras), where compactness is a powerful, frequently used tool.  The functional interpretation has been used extensively to extract quantitative information from such proofs \cite{MR1241250,MR1248130,MR1948057,MR2098088,MR2373327,MR2668247,MR2578604,MR2321766,MR2273533,MR2144170,MR2207130,MR2520390,MR2583826,MR2444453,MR2309276,MR2302929}.

We'll consider just one family of examples, fixed point theorems, with an eye towards the importance of the syntactic form of statements.  A typical fixed point theorem is Edelstein's fixed point theorem \cite{MR0133102}:
\begin{theorem}
Let $(K,d)$ be a compact metric space and $f:K\rightarrow K$ be \emph{contractive}---for any $x,y\in K$, $d(f(x),f(y))<d(x,y)$.  Then for any $x$, the sequence given by $x_0=x$, $x_{k+1}=f(x_k)$ converges to a unique $c$ such that $f(c)=c$.
\end{theorem}

In general, the statement that a sequence converges is not $\Pi_2$---it has the form
\[\forall\epsilon\exists m\forall n(n\geq m\rightarrow d(x_n,x_{n+1})<\epsilon).\]

For Edelstein's theorem, contractivity lets us make the following observation: once we have $d(x_m,c)<\epsilon/2$, we also have $d(x_n,c)<\epsilon/2$ for all $n\geq m$, and in particular, $d(x_n,x_{n+1})<\epsilon$ for $n\geq m$.  If we know what $c$ is, we see that any $m$ with $d(x_m,c)<\epsilon/2$ is an $m$ we are looking for.  The statement
\[\forall\epsilon\exists m\ d(x_m,c)<\epsilon\]
is $\Pi_2$, so we can expect to find explicit bounds for this.  While we can't expect to actually use $c$ when finding bounds, the choice of $c$ is unique, and it has an explicit modulus of uniqueness.  By methods similar to the ones in the effective proof of Jackson's theorem, Kohlenbach showed \cite{MR1241250} that from the modulus of uniqueness, one can construct a modulus of continuity---that is, one can figure out how $c$ varies as $x_0$ varies.  Putting these facts together---the rate of convergence in the original statement depends on the convergence of $x_m$ to $c$, where $c$ is continuous in $x_0$---is enough to find a computable function $N(\epsilon)$ so that for each $x$, each $\epsilon$, and each $n\geq N(\epsilon)$, $d(x_n,c)<\epsilon$ \cite{MR2054493}.  The function $N$ depends on the diameter of $K$ (that is, $\sup_{x,y\in K}d(x,y)$) and a \emph{modulus of contractivity} of $f$---a function $\eta(\epsilon)$ such that $d(x,y)>\epsilon$ implies $d(f(x),f(y))+\eta(\epsilon)<d(x,y)$.

To see that these considerations are really necessary, one can consider the Krasnoselski fixed point theorem:
\begin{theorem}[\cite{MR0068119}]
  Let $K$ be a convex, closed and bounded set in a uniformly convex Banach space $(X,||\cdot||)$, $f$ a mapping of $K$ into a compact subset of $K$ such that $f$ is \emph{non-expansive}---that is, $||f(x)-f(y)||\leq ||x-y||$ for all $x,y$.  Then for every $x_0\in K$, the sequence given by
\[x_{k+1}=\frac{x_k+f(x_k)}{2}\]
converges to a $p$ such that $f(p)=p$.
\end{theorem}
Again this is a convergence statement, so not in a $\Pi_2$ form.  Unlike Edelstein's fixed point theorem, this is really unavoidable---Kohlenbach gives an example \cite{MR1893075} showing that there can be no algorithm finding bounds from $x_0$ and $f$.  However a similar idea is to observe that the quantity $||x_n-f(x_n)||$ is decreasing, and therefore given $\epsilon$ one can hope to find an $n$ such that $||x_n-f(x_n)||\leq\epsilon$.  (This is a bound on the \emph{asymptotic regularity} of the sequence $(x_n)$.)  Asymptotic regularity \emph{is} $\Pi_2$, and it is therefore unsurprising that many explicit bounds are known, both by analytic methods \cite{MR0079667,MR1093989} and by use of the functional interpretation \cite{MR1893075}.

\subsection{Ultraproducts and Similar Constructions}

Another common source of non-constructive proofs is the compactness of first-order logic---that is, nonstandard analysis, or, more generally, the use of ultraproducts\footnote{This general idea has been rediscovered a number of times, especially various special cases that don't depend as heavily on the general logical framework, and is therefore known by a number of names: Banach limits, the Furstenberg correspondence, vague convergence, and graphons.  These are not all exactly the same notion, but when our concern is the extraction of computable bounds, the differences are not significant.}.  We begin with a sequence of models $\mathfrak{M}_N$ of a language $\mathcal{L}$ and combine them into a single model $\mathfrak{M}$ with the property that $\mathfrak{M}\vDash\phi$ for a formula $\phi$ in $\mathcal{L}$ if and only if for ``most'' $N$, $\mathfrak{M}_N\vDash\phi$.  If we prove that $\phi$ must hold in $\mathfrak{M}$, we can conclude that it holds in $\mathfrak{M}_N$ for ``most'', and certainly for infinitely many $N$.  (See \cite{MR3141811} for discussion of the reverse argument, that ultraproducts can be used to show the existence of uniform bounds.)

A typical example is the ergodic-theoretic proof of Szemer\'edi's Theorem \cite{furstenberg77,furstenberg:MR670131}:
\begin{theorem}
For every $\epsilon>0$ and every $k$, there is an $N$ such that if $A\subseteq[1,N]$ with $|A|\geq\epsilon N$ then there are $a,d$ with
\[\{a,a+d,a+2d,\ldots,a+(k-1)d\}\subseteq A.\]
\end{theorem}
Note that this is $\Pi_2$---for every $\epsilon$ and $k$, there is an $N$; the quantifiers over $A$ and over $a,d$ are over finite sets.

If the statement isn't true, we may fix $\epsilon>0$ and $k$, and for each $N$ find an $A_N\subseteq[1,N]$ with $|A_N|\geq\epsilon N$ so that $a$ contains no arithmetic progression of length $k$.  We take the language $\mathcal{L}$ containing a unary relation symbol $\underline{A}$, a unary function $\underline{T}$, and for each formula $\phi(x)$ with only the displayed free variables and each $q\in\mathbb{Q}^{>0}$, a $0$-ary relation symbol $m_{\phi,q}$.  We interpret $([1,2N],A_N)$ as a model of this language by taking $A_N$ as the interpretation of $\underline{A}$, $x\mapsto x+1\mod 2N$ as the interpretation of $\underline{T}$, and
\[([1,N],A_N)\vDash m_{\phi,q}\Leftrightarrow |\{x\mid ([1,N],A_n)\vDash \phi(x)\}|\geq 2qN.\]
(The last clause is technical; it gives us some ability to talk about the measure of sets using formulas in the language.  This is a special case of more general approaches for considering measures in the context of first-order logic \cite{MR2505436,goldbring:_approx_logic_measure,MR819545}.)

We then work in the ultraproduct of these models, which we call $(X,A)$.  Observe that if we can prove, for some $d$, that the ultraproduct satisfies the formula
\[\exists x\, x\in \underline{A}\wedge T^dx\in\underline{A}\wedge\cdots \wedge T^{(k-1)d}x\in\underline{A}\]
then infinitely many of the finite models $([1,2N],A_N)$ satisfy this formula.  Take $N$ much larger than $d$ and let $a\in[1,2N]$ witness this fact.  Since $A_N\subseteq[1,N]$ and each $\underline{T}^{id}a\in A_N$, we must have $\underline{T}^{id}a=a+id$ for each $i<k$.  (We are going to a small amount of trouble here to prevent the case where the progression involves ``wrapping around'', since $T$ is interpreted as addition mod $2N$.)  Therefore $a,a+d,\ldots,a+(k-1)d$ is an arithmetic progression in $A_N$.  This gives the desired contradiction, since the $A_N$ were chosen to be sets with no arithmetic progressions.

Furstenberg gives a proof \cite{furstenberg77} that the ultraproduct satisfies this formula by way of interpreting the ultraproduct as dynamical system.  (He phrases his construction quite differently, but the essential idea is the same.)

If we want to give a constructive version of this proof, we face the following obstacle.  The ergodic theoretic argument involves statements about the measure---say, $\mu(A\cap TA)>0$.  Translated into a formula, this is
\[\exists q\in\mathbb{Q}^{>0} m_{Ax\wedge A(Tx),q}.\]
We need to distinguish between formulas in the precise, formally defined language $\mathcal{L}$ and formulas in the informal sense we used them in the previous sentence.  This formula, with its quantifier over $\mathbb{Q}^{>0}$, is a formula in the broader sense, but it isn't actually a formula in $\mathcal{L}$.  In particular, we don't automatically have that $(X,A)$ satisfies this formula exactly when most $([1,2N],A_N)$ satisfy this formula.

There is an obvious attempt at a solution: we could extend $\mathcal{L}$ to a bigger language $\mathcal{L}'$, with two sorts, where the second sort is intended to represent $\mathbb{Q}^{>0}$.  We would replace $m_{\phi,q}$ with $m_{\phi}(q)$, where $q$ will obviously range over the second sort.  Then we start with two-sorted models $([1,2N],A_N,\mathbb{Q}^{>0})$.  The problem is that when we take the ultraproduct, we also have to take the ultraproduct in the second sort, so we get a model $(X,A,(\mathbb{Q}^{>0})^*)$, where $(\mathbb{Q}^{>0})^*$  is the (positive) \emph{nonstandard} rational numbers---which includes infinitesimal rationals.

But the formula we want, $\exists q\in\mathbb{Q}^{>0} m_{Ax\wedge A(Tx)}(q)$, is not the same as the formula $\exists q\in(\mathbb{Q}^{>0})^* m_{Ax\wedge A(Tx)}(q)$; the latter allows for the posibility that the measure of $A\cap TA$ is ``positive'' but infinitesimal.

The issue is that when we talk about the distinction between sets of positive measure and sets of measure $0$, we want to work with the non-compact domain $\mathbb{Q}^{>0}$.  In an ultraproduct, all the domains in the language of the model are compact, so we can only discuss the distinction between positive and $0$ measure by including quantifiers outside the language of the ultraproduct.

However the \emph{transfer principle} tells us that we have the following correspondence:
\begin{lemma}
  Consider a statement in the form $\forall x\in D\exists y\in S\phi(x,y)$ where $D$ is arbitrary, $S$ is countable, and $\phi$ is a statement in the language $\mathcal{L}$.  Then $\mathfrak{M}\vDash\forall x\exists y\phi(x,y)$ iff for every $x\in D$ there is a $y\in S$ such that for most $N$, $\mathfrak{M}_N\vDash\phi(x,y)$.
\end{lemma}
In other words, we have a correspondence between $\Pi_2$ sentences in $\mathfrak{M}$ and those in the $\mathfrak{M}_N$, namely that $\Pi_2$ statements which are \emph{uniformly} true in the finite models are true in the infinite one.

For statements which are not $\Pi_2$, we do not have such a correspondence.  For example, the ergodic proof of Szemer\'edi's Theorem uses the mean ergodic theorem, the relevant case of which is the following statement:
\begin{quote}
Let $\alpha_n(x)=\frac{1}{n}\sum_{i\leq n}\chi_{A}(x+n)$ and let $||\alpha(x)||=\sqrt{\frac{1}{|X|}\sum_{x\in X}\alpha^2(x)}$.  For every $\epsilon>0$, there is an $n$ such that for all $m\geq n$
\[||\alpha_n(x)-\alpha_m(x)||<\epsilon.\]
\end{quote}
We may treat the inner part of the statement---$||\alpha_n(x)-\alpha_m(x)||<\epsilon$---as being a formula in $\mathcal{L}$ (working this out in detail requires expanding the language and getting into technicalities about representing the measure as part of $\mathcal{L}$).  Even so, this isn't $\Pi_2$, since we have three quantifiers over countable domains on the outside.

In fact, this statement is true in the finite models, for utterly trivial reasons---in $([1,2N],A_N)$, take $n$ to be much, much larger than $N$, so that for any $m\geq n$, $\alpha_m(x)$ is the function which is very close to $\frac{1}{N}\sum_{i\leq N}\chi_A(i)$ at every point.  This is a vacuous argument, and it doesn't reflect the real mathematical content of the mean ergodic theorem.  We can't expect to express the full content of saying that an infinite sequence converges in a finite model.

The functional interpretation tells us that every statement implies a $\Pi_2$ statement which captures its computational content.  The right statement here is
\begin{quote}
  For every $\epsilon>0$ and every function $F$, there is an $n$ such that if $F(n)\geq n$ then $||\alpha_n-\alpha_{F(n)}||<\epsilon$.
\end{quote}
This is also true in the finite models, with a bound on $n$ that depends on $\epsilon$ and $F$, but \emph{not} on the size of the finite model.  This is the statement about finite models which matches up correctly with the mean ergodic theorem in the ultraproduct.

We can make this statement slightly nicer with the following observation: given a function $F$, define $F'$ by
\[F'(n)=\left\{
  \begin{array}{ll}
    \text{The least }m\in[n,F(n)]\text{ such that }||\alpha_n-\alpha_m||\geq\epsilon&\text{if there is one}\\
F(n)&\text{otherwise}
  \end{array}\right..
\]
Then $F'$ is computable from $F$, and by appling the statement above to $F'$ instead of $F$, we obtain
\begin{quote}
  For every $\epsilon>0$ and every function $F$, there is an $n$ such that if $F(n)\geq n$ then for all $m\in[n,F(n)]$, $||\alpha_n-\alpha_{m}||<\epsilon$.
\end{quote}
The usual mean ergodic theorem tells us that the averages $\alpha_n$ stabilize.  This version tells us that we can find arbitrarily long intervals on which the average remains stable.  This is known as the \emph{metastability} of the average \cite{avigad:MR2550151,tao:MR2408398,MR2912715}.

A recent generalization of this idea is in \cite{towsner:abs_cont}, where the corresponding notion of metastability for a double limit is used to give an effective version of a theorem about $L^1$ functions; this is applied in \cite{towsner:banach} to give explicit bounds for a theorem about Banach spaces---the ``failure of local unconditionality of the James space''---whose usual proof involves an ultraproduct.

Note that in this case it's incidental that what the functional interpretations extracts from proofs which used ultraproducts is computable---the same method would work with any language and any replacement for our finite models.  What's really important is that the functional interpretation tells us how to correspond statements about the ultraproduct model to statements in the original models.  In almost every situation where we care about this, it makes sense to view us as extracting information computable from the original models, but there has been some recent investigation of this idea more abstractly, where one has the right syntactic features to extract information using the functional interpretation, but where one need not be working with computable information \cite{MR2964881,MR2200579,MR3145191}

\section{Further Reading}

There are two well-established and complementary introductions to the functional interpretation.  Avigad and Feferman \cite{avigad:MR1640329} give a thorough introduction to the formal theory of the functional interpretation in the most important settings, including a number of variants and many applications within proof theory.  Kohlenbach \cite{kohlenbach:MR2445721} gives a detailed guide to the applications of the functional interpretation outside logic, especially in analysis.


\bibliographystyle{plain}
\bibliography{../../Bibliographies/main}
\end{document}